\documentclass{article}
\usepackage{amsmath}
\usepackage{amsthm}
\usepackage{amssymb}
\usepackage{amsfonts}
\usepackage{graphicx}
\usepackage{mathtools}
\usepackage{verbatim}
\usepackage{pgf}
\usepackage{stmaryrd}
\usepackage{amscd}
\usepackage[mathscr]{euscript}
\usepackage{graphicx}
\usepackage{tikz-cd}
\usepackage{tikz}
\usetikzlibrary{decorations.pathmorphing}
\usepackage{xfrac}
\usepackage[utf8]{inputenc}
\usepackage[verbose]{wrapfig}
\usepackage{geometry} 
\usepackage{fancyhdr}
\usepackage[bf,small, center]{titlesec}
\usepackage{multirow}
\usepackage{enumitem}
\usepackage{scalerel} 
\DeclarePairedDelimiter{\ceil}{\lceil}{\rceil}
\DeclarePairedDelimiter{\floor}{\lfloor}{\rfloor}
\usepackage{cancel} 
\usepackage{bbm}
\usepackage{todonotes}
\usepackage[english]{babel}
\usepackage{csquotes}
\usepackage[style=alphabetic,
			backend=biber,
			url=false,
			isbn=false,
			doi=false,
			maxnames=30]{biblatex}
\usepackage{tocbibind}
\usepackage{mathpazo} 
\PassOptionsToPackage{hyphens}{url}\usepackage[hidelinks]{hyperref}

\setitemize{noitemsep}

\begingroup
		\makeatletter
		\@for\theoremstyle:=definition,remark,plain\do{%
				\expandafter\g@addto@macro\csname th@\theoremstyle\endcsname{%
						\addtolength\thm@preskip\parskip
						}%
				}
\endgroup

\newcommand{\R}{{\mathbb{R}}}
\newcommand{\Z}{{\mathbb{Z}}}

\newcommand{\sph}{{\mathbb{S}}}
\newcommand{\E}{{\mathbb{E}}}
\newcommand{\F}{{\mathbb{F}}}

\newcommand{\cat}{{\mathcal{C}}}
\newcommand{\Cat}{{\mathscr{C}\mathrm{at}}}

\newcommand{\Spc}{\catname{Spc}}
\newcommand{\AbGrp}{\mathrm{Ab}}

\newcommand{\Alg}{\mathrm{Alg}}
\newcommand{\Mon}{\mathrm{Mon}}

\newcommand{\id}{\mathrm{id}}
\newcommand{\ev}{\mathrm{ev}}
\newcommand{\coev}{\mathrm{coev}}
\newcommand{\tors}{\mathrm{tors}}
\newcommand{\op}{\text{op}}

\newcommand{\loops}{\Omega}
\newcommand{\Tor}{\mathrm{Tor}}
\newcommand{\Ext}{\mathrm{Ext}}

\newcommand{\surj}{\twoheadrightarrow}

\newcommand{\Spectra}{\mathrm{Sp}}
\newcommand{\stable}{\mathrm{st}}

\newcommand{\Fun}{\mathrm{Fun}}

\newcommand{\Perf}{\mathscr{P}\mathrm{erf}}

\renewcommand{\Spc}{\mathrm{Spc}}

\newcommand{\Vect}{\mathrm{Vect}}

\newcommand{\Hom}{\mathrm{Hom}}
\DeclareMathOperator{\Ind}{Ind}

\renewcommand{\P}{\mathbb{P}}

\newcommand{\StMod}{\mathrm{StMod}}
\newcommand{\StExt}{\mathrm{StExt}}

\newcommand{\Coh}{\mathrm{Coh}}
\newcommand{\Sq}{\mathrm{Sq}}
\newcommand{\cofib}{\mathrm{cofib}}
\newcommand{\fib}{\mathrm{fib}}
\newcommand{\Nm}{\mathrm{Nm}}
\newcommand{\gr}{\mathrm{gr}}
\newcommand{\Gr}{\mathrm{Gr}}
\newcommand{\pol}{\mathrm{pol}}
\newcommand{\coker}{\mathrm{coker}}

\newcommand{\coBar}{\mathrm{coBar}}
\newcommand{\twist}{\mathsf{tw}} 

\newcommand{\inj}{\hookrightarrow}

\DeclareMathOperator*{\colim}{colim}

\DeclareMathOperator{\Mod}{Mod}
\DeclareMathOperator{\LMod}{LMod}

\DeclareMathOperator{\ho}{ho}

\numberwithin{equation}{subsection}

\theoremstyle{plain} \newtheorem{theorem}[equation]{Theorem}
\theoremstyle{plain} \newtheorem*{theorem*}{Theorem}
\theoremstyle{definition} \newtheorem{defn}[equation]{Definition}
\theoremstyle{plain} \newtheorem{prop}[equation]{Proposition}
\theoremstyle{plain} \newtheorem*{prop*}{Proposition}
\theoremstyle{plain} \newtheorem{lemma}[equation]{Lemma}
\theoremstyle{plain} \newtheorem{cor}[equation]{Corollary}
\theoremstyle{definition} \newtheorem{ex}[equation]{Example}
\theoremstyle{definition} \newtheorem{exs}[equation]{Examples}
\theoremstyle{definition} 
\theoremstyle{definition} \newtheorem{rmk}[equation]{Remark}
\theoremstyle{remark} 
\theoremstyle{definition} \newtheorem{idea}[equation]{Idea}
\theoremstyle{definition} \newtheorem{obs}[equation]{Observation}
\theoremstyle{plain} 
\theoremstyle{remark} 
\theoremstyle{definition} 
\theoremstyle{definition} 
\theoremstyle{remark} \newtheorem{ntn}[equation]{Notation}
\theoremstyle{plain} \newtheorem{goal}[equation]{Goal}
\theoremstyle{remark} 
\theoremstyle{definition} 
\theoremstyle{definition} \newtheorem{cons}[equation]{Construction}
\theoremstyle{definition} 
\theoremstyle{remark} \newtheorem{warning}[equation]{Warning}
\theoremstyle{definition} \newtheorem{recollection}[equation]{Recollection}

\titleformat{\subsection}[runin]{\bfseries}{\thesubsection}{1em}{}

\newif\iflong
\longfalse 
\newcommand{\inLongVersion}[2]{\iflong #1 \else #2 \fi}
\newif\ifinclex
\inclextrue

\title{Categorical dynamics on stable module categories}
\author{Lucy Yang}
\date{\today}

\begin{document}
\maketitle

\begin{abstract}
	Let $ A $ be a finite connected graded cocommutative Hopf algebra over a field $ k $.  
	There is an endofunctor $ \twist $ on the stable module category $ \StMod_A $ of $ A $ which twists the grading by $ 1 $. 
	We show the categorical entropy of $ \twist $ is zero.
	We provide a lower bound for the categorical polynomial entropy of $ \twist $ in terms of the Krull dimension of the cohomology of $ A $, and an upper bound in terms of the existence of finite resolutions of $ A $-modules of a particular form. 
	We employ these tools to compute the categorical polynomial entropy of the twist functor for examples of finite graded Hopf algebras over $ \F_2 $.
\end{abstract}

\tableofcontents

\section{Introduction}
\subsection{On growth and complexity.}
Homotopy theory abounds with quantitative estimates. 
\begin{itemize}
	\item James \cite[Corollary 1.22]{MR83124} (resp. Toda \cite[Theorem 8.11]{MR92968}) showed that the $ p $-primary torsion of $ \pi_* S^{2n+1} $ is annihilated by $ 2^{2n} $ for $ p = 2 $ (resp. by $ p^{2n} $ for $ p $ odd). 
	\item Cohen, Moore, and Neisendorfer showed that for an odd prime $ p > 3 $ and $ n \geq 3 $, the homotopy groups of $ \loops^2(S^n/p^r) $ are annihilated by $ p^{2r+1} $ \cite[Proposition 2.5]{MR921471}.  
	\item Mathew showed that the $ p $-primary torsion of the kernel of the Hurewicz homomorphism is annihilated by $ p^{\ceil*{\sfrac{\ell(n)}{2p-2}}} $ where $ \ell $ is a linear function in $ n $ \cite[Theorem 1.3]{MR3493414}, refining work of Arlettaz \cite[Theorem 4.1]{MR1402323}.
\end{itemize}
A central theme among these results is \emph{uniform rates of growth}: James and Toda's (resp. Arlettaz and Mathew's) results identify an exponential dependence of order of $ \pi_* $ on dimension (resp. Postnikov tower stage).  
On the other hand, Cohen-Moore-Neisendorfer's results are constant with respect to dimension. 

The precise study and quantification of growth rates and complexity is no stranger to dynamical systems. 
Given a self-map $ f:X \to X $ of a compact topological space, its \emph{topological entropy} $ h_\mathrm{top}(f) \in \R_{\geq 0} \cup \{\infty \} $ is a measure of its asymptotic expansion. 
Topological entropy detects geometric properties of a system: For instance, let $ SM $ be the unit tangent bundle of a Riemannian manifold $ (M, g) $.
The manifold $ SM $ has a self-map $ \phi $ given by time 1 geodesic flow, and the entropy $ h_\mathrm{top}(\phi) $ is related to the sectional curvature of $ g $ \cite[Theorem 2]{Mann_more}. 
Beyond measuring geometric quantities of a dynamical system, topological entropy admits a wealth of comparison theorems \cites[2.3 Corollaire]{MR2026895}[Remark 1.11 on p.350]{MR1792240}[Theorem 1.1]{MR889979}. 
Such results formalize intuitive notions such as: \emph{a dynamical system must be at least as complex as a quotient system or its linearization}. 

In view of the preceding discussion, asymptotic growth rates are both interesting invariants in their own right in addition to being a useful way of organizing information, and quantitative homotopy theory has much to gain from a dynamical perspective on asymptotic growth rates. 
But how should one reconcile an analytic approach to growth and complexity with the study of stable $ \infty $-categories that comprise the domain of homotopy theory? 

\subsection{Categorical dynamics and homotopy theory.}
The categorical dynamics of Dmitrov--Haiden--Katzarkov--Kontsevich \cite{MR3289326} offers a path forward.
\begin{defn}
	A \emph{categorical dynamical system} consists of a small stable $ \infty $-category $ \cat $ satisfying a suitable finiteness condition\footnote{In topological dynamical systems, one typically asks for the underlying topological space to be compact. 
	The relevant finiteness condition we impose here is that $ \cat $ must have a single compact generator.}  
	and an exact endofunctor $ F: \cat \to \cat $. 
\end{defn}
Given two objects $ X, Y \in \cat $, the \emph{complexity} $ \delta(X,Y) $ (Definition \ref{defn:complexity}) of $ Y $ with respect to $ X $ measures the size of $ Y $ relative to $ X $ by counting the minimal length of a filtration of $ Y $ by shifts of $ X $. 
Given a categorical dynamical system $ (\cat, F: \cat \to \cat) $ and a generator $ G $ of $ \cat $, the categorical entropy $ h_{\mathrm{cat}} (F) \in \R \cup\{\pm \infty\} $ of $ F $ is the exponential growth rate of $ \delta(G, F^n(G)) $ as $ n \to \infty $. 

While the categorical entropy measures exponential growth rates, it is designed so as not to see sub-exponential growth \cites[2.13]{MR3289326}. 
The categorical \emph{polynomial} entropy $ h_\pol(F) $ is sensitive to sub-exponential growth, and is related to the study of polynomial dynamical degrees \cites[Corollary 5.9]{MR4233273}{MR4196396}[\S1.3]{MR4030548}. 
Collectively, these notions of entropy \emph{quantify} fuzzy notions such as complexity and growth.

\begin{goal}
	Use categorical entropy and its polynomial variant to describe complexity and growth in stable homotopy theory. 
\end{goal}
Certain algebraic objects occupy a special place in homotopy theory--their categories of representations are at once good approximations to more topological categories, yet more amenable to algebraic manipulations than their topological counterparts \cite{MR4405748,MR2640996,MR960945,Lurie:2010,MR3356769}. 
Finite-dimensional connected graded cocommutative Hopf algebras (Definition \ref{defn:hopf_alg}) are one such example, arising as algebras of cohomology operations (Example \ref{ex:mod2steenrod}) and as homological manifestations of connected finite-dimensional topological groups (Example \ref{ex:liegp_homology}). 

Given such a Hopf algebra $ A $ over a field $ k $, consider the category of $ A $-representations on perfect $ k $-modules $ \Mod_A(\Perf_k) $. 
We localize away from the perfect $ A $-modules $ \Perf_A \subset \Mod_A(\Perf_k) $ to obtain the \emph{stable module category} $ \StMod_A $ of $ A $ (Definition \ref{defn:stmod}). 
There is a dynamical system on $ \StMod_A $ which twists the grading by 1 (Construction \ref{cons:twist_dynamicalsystem}).  
\begin{theorem} \label{thm:mainthm} 
	Let $ A $ be a finite-dimensional connected graded cocommutative Hopf algebra over a field $ k $. 
	Consider the endofunctor $ \twist $ of the stable module category of $ A $ given by twisting the grading by 1:
	\begin{equation*}
		\twist: \StMod_A \to \StMod_A .
	\end{equation*} 
	\begin{enumerate}
		\item The categorical entropy of the twist functor $ \twist $ satisfies
		\begin{equation*}
			h_{\mathrm{cat}}(\twist) =0 .
		\end{equation*} \vspace{-1em}
		\item The categorical polynomial entropy $ h_\pol $ of the twist functor admits the following bounds
		\begin{enumerate}
			\item \label{thmitem:lowerboundcatpolent}  
				{\center $h_\pol(\twist) \geq \mathrm{Krull\:dim\:} H^*(A;k) - 1.$}
			\item \label{thmitem:upperboundcatpolent} Suppose given a tower under $ A $ (Definition \ref{defn:hopfalg_tower}) with $ \ell+1 $ storeys.  
			\begin{equation*}
				h_\pol(\twist) \leq \ell .
			\end{equation*}
		\end{enumerate}
	\end{enumerate}
\end{theorem}
These are given by Theorems \ref{thm:twist_expentropy_iszero} and \ref{thm:catpol_bounds}.
A method for obtaining sharper upper bounds for $ h_\pol(\twist) $ is suggested by Corollary \ref{prop:upperboundres_length_refined}. 
We will explore applications of Yomdin--style estimates \cite[\S4]{MR4233273} to understanding $ h_\pol(\twist) $ in a sequel. 
\begin{rmk}
	We use the terminology `twist' to refer to changes in the index of a graded object and `shift' or `suspension' for a homological degree--or in the terminology of stable $ \infty $-categories, suspension or desuspension. 
\end{rmk}
We expect the inequality in Theorem \ref{thm:mainthm}(\ref{thmitem:lowerboundcatpolent}) is always an equality, and this is the subject of ongoing work.  
The proof of Theorem \ref{thm:mainthm}(\ref{thmitem:lowerboundcatpolent}) utilizes linear-algebraic bounds on categorical entropy (\S\ref{subsection:linalg_bounds_ent}), which relies on knowing that $ \StMod_A $ is proper. 
While $ \StMod_A $ is proper, it is not known to be smooth, and Theorem \ref{thm:mainthm}(\ref{thmitem:upperboundcatpolent}) is not amenable to a similar reduction. 

To circumvent this issue, we work directly at the level of stable $ \infty $-categories--in particular, we rely on a careful analysis of resolutions in $ \StMod_A $. 
Instead of reducing the computation of entropy to cohomological computations, we produce specific resolutions in $ \StMod_A $ which allow us to bound the complexity via a downward induction. 
In particular, a good choice of resolution leads to better approximations to $ h_\pol(\twist) $ (see Observation \ref{obs:choiceoftower} and the examples of \S\ref{subsection:noncomm_polentropy_computed}).
In proving our result, we exhibit the utility of a homotopical perspective on categorical entropy.  
\begin{rmk}
	While the stable module category and the functor $ \twist $ can be defined for the category of modules over any graded Frobenius algebra over $ k $, our assumption that $ A $ is a finite connected graded cocommutative Hopf algebra over $ k $ guarantees that $ H^*(A; k) $ is both a graded-commutative ring (Proposition \ref{prop:hopf_coh_gradedcomm}) and has finite Krull dimension by a general result of Wilkerson \cite[Theorem A]{MR597872}. 
\end{rmk}

\subsection{The genesis of this paper.}
Theorem \ref{thm:mainthm} is inspired by parallels between $ \StMod_A $ and the bounded derived category of coherent sheaves on projective space $ D^b(\Coh_{\P^n}) $ and computations of Dmitrov-Haiden-Katzarkov-Kontsevich and Fan-Fu-Ouchi. 

The categorical entropy and the categorical polynomial entropy of twists on coherent sheaves on projective space is well-understood; the following are special cases of \cite[Proposition 6.4]{MR4233273} and \cite[Lemma 2.13]{MR3289326}.
\begin{prop}\label{prop:entropy_twistcohproj_inspiration}
	Let $ k $ be a field. 
	Write $ \P^n $ for projective space of dimension $ n $ over $ k $. 
	Consider the autoequivalence $ - \otimes \mathcal{O}(1):D^b\left(\Coh_{\P^n }\right) \to D^b\left(\Coh_{\P^n }\right) $. 
	\begin{enumerate}
		\item The categorical entropy of the autoequivalence $ - \otimes \mathcal{O}(1) $ is given by $ h_{cat}(- \otimes \mathcal{O}(1)) = 0 $.
		\item \label{propitem:inspiration_hpol} The categorical polynomial entropy of the autoequivalence $ - \otimes \mathcal{O}(1) $ is given by
		\begin{equation*}
			h_\pol(- \otimes \mathcal{O}(1)) = n .
		\end{equation*}
	\end{enumerate}
\end{prop}
The stable module category $ \StMod_A $ of a finite connected graded cocommutative Hopf algebra is a homotopical analogue of the bounded derived category $ D^b\left(\Coh_{\P^n }\right) $.  
To motivate this analogy, consider the following result due to Bern\v{s}te\u{\i}n--Gel{\textquotesingle}fand--Gel{\textquotesingle}fand \cite{MR509387}; we follow the presentation of \cite[\S4.3.6]{MR1438306}. 

Let $ k $ be a field and $ V $ an $ (n +1) $-dimensional vector space over $ k $.  
Endow the exterior algebra $ \Lambda:= \bigwedge_k V $ with the canonical graded-commutative ring structure, where $ V $ is placed in grading $ 1 $. 
Any graded $ \Lambda $-module $ E $ and $ v \in V $ is equipped with morphisms $ \cdot v: E_i \to E_{i+1} $ for all $ i $. 
Considering these morphisms for all $ v \in V $ gives morphisms $ E_i \to E_{i+1} \otimes \mathcal{O}(1) $ which we assemble into a functor defined on discrete graded $ \Lambda $-modules  
\begin{equation}\label{eq:stmod_Dbcoh}
\begin{tikzcd}[row sep=0ex,column sep=small]
	\Mod^{\heartsuit}_{\Lambda} \ar[r] & D^b\left(\Coh_{\P(V)}\right) \\
	V = \oplus V_i \ar[r, mapsto] & \colim \left( \cdots \to V_j \otimes \mathcal{O}(j) \to V_{j+1} \otimes \mathcal{O}(j+1) \to \cdots \right).
\end{tikzcd}
\end{equation}
\begin{prop}
	\cites[\S4.3.6]{MR1438306}[]{MR509387} \label{prop:stmod_Dbcoh}
	The functor (\ref{eq:stmod_Dbcoh}) descends to an equivalence of stable $ \infty $-categories
	\begin{align*}
		\StMod_{\Lambda} \to D^b\left(\Coh_{\P(V)}\right) 
	\end{align*}
	where the homotopy category of $ \StMod_{\Lambda} $ has the same objects as $ \Mod^{\heartsuit}_{\Lambda} $ but the morphism spaces are subject to the equivalence relation $ \varphi,\psi: V \to W \in \Mod^{\heartsuit}_{\Lambda} $ are identified if their difference factors through a projective $ {\Lambda} $-module (see Definition \ref{defn:stmod} and following discussion). 
\end{prop}
We observe that the algebra $ \Lambda = {\bigwedge_k V}$ is an example of a finite-dimensional connected graded cocommutative Hopf algebra over $ k $. 
We extend this correspondence by observing
\begin{itemize}
	\item The twist by 1 functor $ \twist $ which acts on a graded module $ M_* $ via $ M(1)_\ell = M_{\ell-1} $ is sent under the inverse equivalence of Proposition \ref{prop:stmod_Dbcoh} to $ -\otimes \mathcal{O}(1) $. 

	\item We can extract the \emph{dimension} $ \dim_k V -1 $ of our projective space from the left-hand side as one less than the Krull dimension of the cohomology $ H^*(\Lambda; k) $ (Definition \ref{defn:hopf_alg_(co)homology}).
\end{itemize}
\begin{rmk}
	The analogy between stable module categories and categories of sheaves on projective space is taken further in \cite[\S 9.2]{Mat15}.
\end{rmk}
In view of Proposition \ref{prop:stmod_Dbcoh} and Example \ref{ex:exteriorhopfalg_hpol_computed}, Theorem \ref{thm:mainthm} and Proposition \ref{prop:upperboundres_length_refined} partially generalize Proposition \ref{prop:entropy_twistcohproj_inspiration}: While Fan--Fu--Ouchi are able to demonstrate equality in Theorem \ref{prop:entropy_twistcohproj_inspiration}(\ref{propitem:inspiration_hpol}), the upper bound we obtain in Theorem \ref{thm:mainthm}(\ref{thmitem:lowerboundcatpolent}) is weaker. 
The proof of the results \cite[\S6]{MR4233273} which are generalized here uses essentially that the bounded derived category of coherent sheaves $ D^b \left(\Coh_{\P^n}\right) $ on a smooth projective variety is itself smooth and proper as a stable $ \infty $-category (Definition \ref{def:smoothpropcat}). 
While $ \StMod_A $ is proper, it is not known to be smooth. 
The novelty of our work lies in the development and application of tools used to obtain such a statement--in particular, a careful analysis of resolutions in $ \StMod_A $. 

\subsection{Outline.} 
We collect relevant background on categorical polynomial entropy in \S\ref{section:catentropy}. 
Next, we introduce and collect the requisite background and relevant properties of Hopf algebras in \S\ref{section:hopfalgebras}. 
In particular, we define the stable module category of a Hopf algebra in \S\ref{subsection:stablemodcats}. 
We arrive at a proof of Theorem \ref{thm:mainthm} in \S\ref{section:twist_entropy_results}.
Finally, we compute the categorical entropy and the categorical polynomial entropy of the twist functor for many examples of relevant Hopf algebras in \S\ref{section:examples}. 

\subsection{Notation \& Conventions.}
We assume some familiarity with stable $ \infty $-categories and homological algebra. 
The reader who is unfamiliar with stable $ \infty $-categories should feel free to substitute the term ``pre-triangulated $ k $-linear dg category'' or ``stable model category'' for ``stable $ \infty$-category'' in the following by a result of Cohn \cite[Corollary 5.5]{Cohn}. 

We use the conventions of stable homotopy theory--e.g., if $ R $ is a discrete ring, $ \Mod_R $ is used to refer to the derived $ \infty $-category of module spectra over the Eilenberg-Mac Lane ring spectrum $ HR $ associated to $ R $. 
The additive 1-category of discrete modules over $ R $ is denoted by $ \Mod_R^{\heartsuit} $ (with special cases such as projective modules and vector spaces automatically assumed to be discrete). 
Given an $ \infty $-category $ \cat $, we write $ \ho \cat $ for the (triangulated) \emph{homotopy (1-)category of $ \cat $}, with morphism sets given by $ \hom_{\ho \cat}(X, Y) = \pi_0 \hom_\cat(X, Y) $. 

All tensor products and hom are to be taken in the derived sense, unless stated otherwise. 
We use the notation $ \Ext_A^{-s,-t}(k,M) := \pi_{s,t}\hom_{A}(k,M) = H^{-s,-t}(A; M) $ interchangeably. 

We attempt to follow existing conventions whenever possible, though this inevitably leads to potentially ambiguous terminology: A graded $ k $-vector space $ A $ is said to be \emph{connected} if it is concentrated in nonnegative grading and $ A_0 \simeq k $ (as opposed to being connective with respect to a t-structure on $ \Perf_k^\gr $). 

If $ A $ is a graded ring, we eventually write $ \Mod_A $ for \emph{graded} left $ A $-modules. 
All algebras are taken to be associative unless otherwise specified. 
Given an associative algebra $ A $, a \emph{module} over $ A $ will refer to a left module. 
Given an element $ a \in A $ and a left $ A $-ideal $ J \subseteq A $, we will write $ \cdot a : J \to J $ for the map which multiplies by $ a $ on the right. 

\subsection{Acknowledgements.}
The author thanks Michael Hopkins for invaluable guidance and Laura DeMarco for (inadvertently) starting her down this path. 
The author benefited from conversations with Ben Antieau, Laurent Côté, Paul Goerss, and Kevin Lin. 
The author thanks Araminta Amabel, Elden Elmanto, Peter Haine, and Noah Riggenbach for detailed comments on an early draft. 
The author gratefully acknowledges support from the NSF Graduate Fellowship Research program under Grant No. 1745303. 

\section{Categorical entropy}\label{section:catentropy}
All results contained in this section can be found in \cite{MR3289326} and \cite{MR4233273}; we make no claim to originality. 
\subsection{Complexity \& its asymptotic growth rates.}\label{subsection:entropy_defn}
To motivate the definition of categorical entropy, we recall the definition of entropy \cite[\S7.4]{Walters} of a continuous endomorphism $ f :X \to X $ of a (compact) topological space $ X $. 
Given two open covers $ \mathcal{U}, \mathcal{V} $ of $ X $, one can measure whether they are `independent' (or the complexity of $ \mathcal{V} $ with respect to $ \mathcal{U} $) using Shannon information. 
Any self-map $ f:X \to X $ acts on open covers $ \mathcal{U} $ of $ X $ via pullback.
We can measure the complexity of $ f $ by `testing' on a finite open cover $ \mathcal{U} $: the entropy $h(f; \mathcal{U}) $ of $ f $ \emph{with respect to $ \mathcal{U} $} is measured by the independence of the open covers $ \mathcal{U},  f^{-1}\mathcal{U}, \ldots, f^{-n}\mathcal{U} $ as $ n \to \infty $. 
Then one defines the topological entropy of $ f $ to be the supremum of $h(f; \mathcal{U}) $ over all finite open covers $ \mathcal{U} $ of $ X $.

In the setting of a small stable $ \infty $-category $ \cat $, a generator plays the role of a test object.  
\begin{recollection}\label{recollection:thick_subcategory}
	Given a small stable $\infty$-category $ \cat $ and a set $ (X_i)_{i \in I} $ of objects in $ \cat $, the \emph{thick subcategory} $ \langle X_i \rangle_{i \in I} $ generated by the $ X_i $ is the smallest stable subcategory of $ \cat $ containing each of the $ X_i $ and closed under finite (co)limits and retracts. 
	Note that if $ Y \in \langle X_i \rangle $, then $ \langle X_i,Y \rangle = \langle X_i \rangle  $. 
	We say that an object $ X \in \cat $ in a stable $ \infty $-category is a \emph{generator} if $ \langle X \rangle = \cat $. 
\end{recollection}
We begin by introducing a measure of relative complexity between \emph{objects} in a stable $ \infty $-category, then define the entropy of an endofunctor $ F $ in terms of the relative complexity of $ F $ evaluated on our test object. 
\begin{defn}\label{defn:complexity} \cite[Definition 2.1]{MR3289326}
	Let $ \cat $ be a stable $ \infty $-category. 
	Given two objects $ X , Y \in \cat $, we define the \emph{complexity $ \delta(X, Y) $ of $ Y $ with respect to $ X $} to be the infimum
	\begin{equation*}
		\delta(X, Y) = \inf \left\{ \ell \middle|\;
		\begin{matrix} 
			\text{there exist objects equipped with a filtration}\; 0 = Y_0 \to Y_1 \to \cdots \to Y_\ell\\
			 \text{such that } \cofib(Y_{i-1} \to Y_{i}) \simeq X[n_i] \; \text{for all $ i $ and $ Y$ is a retract of $ Y_\ell$} 
		\end{matrix} \right\} 
	\end{equation*}
	if $ Y $ is in the thick subcategory generated by $ X $, or $ \infty $ otherwise. 
\end{defn}
Note that our definition of $ \delta(X, Y) $ agrees with the complexity of Dimitrov--Haiden--Katzarkov--Kontsevich \cite[Definition 2.1]{MR3289326} in the homotopy triangulated categoy $ \ho \cat $. 
\begin{rmk}
	We restrict our focus to the value of the \emph{complexity function} of \cite[Definition 2.1]{MR3289326} at $ t = 0 $. 
	See also Remark \ref{rmk:entropyfunctionorvalue}. 
\end{rmk}
\begin{obs}\label{obs:deltazeroshiftinsensitive}
	The complexity $ \delta(X, Y) $ counts the \emph{number} of copies of (shifts of) $ X $ in a minimal filtration of $ Y $.   
	The reader may choose to think of $ \delta(X, Y) $ as the size of $ Y $ relative to $ X $.
	In particular, $ \delta(X, Y) $ is insensitive to (de)suspension, i.e. $ \delta(X, \Sigma Y) = \delta(X, Y) $.
\end{obs}
\begin{lemma}\label{lemma:complexityprops} \cite[Proposition 2.2]{MR3289326}
	Let $ \cat $ be a stable $ \infty $-category. 
	The complexity $ \delta(-,-) $ satisfies the properties
\begin{enumerate}
	\item \label{lemmaitem:complexity_triangleineq} For any $ X, Y, Z \in \cat $, $\delta(X, Z) \leq \delta(X, Y) \cdot \delta(Y, Z) $
	\item \label{lemmaitem:complexity_subadditive} Given any cofiber sequence $ X \to Y \to Z $ and any object $ W $, $ \delta(W, Y) \leq \delta(W, X) + \delta(W, Z) $
	\item Given any exact functor $ F: \cat \to \mathcal{D} $ and $ X, Y \in \cat $, $ \delta(X, Y) \geq \delta(F(X), F(Y)) $.
\end{enumerate}
\end{lemma}
Next we introduce a notion which will allow us to generalize Lemma \ref{lemma:complexityprops}(\ref{lemmaitem:complexity_subadditive}). 
The key point is the
\begin{idea}
	The complexity of an object $ X $ can be controlled by the total complexity of a finite resolution of $ X $.
\end{idea} 
We begin by axiomatizing what we mean by `resolution.'
\begin{defn}\label{defn:resolution}
	Let $ \cat $ be a stable $ \infty $-category and suppose given a finite sequence of composable morphisms 
	\begin{equation}\label{eq:resolution} 
		X_n \xrightarrow{f_n} X_{n-1} \xrightarrow{f_{n-1}} \cdots X_{0} \xrightarrow{f_0} X_{-1} 
	\end{equation} 
	in $ \cat $.
\begin{enumerate}
	\item The sequence (\ref{eq:resolution}) is a \emph{chain complex} if each $ f_{i-1} \circ f_i $ is nullhomotopic. 
	\item Suppose given a sequence (\ref{eq:resolution}) which is a {chain complex}. We define recursively
	\begin{enumerate}
		\item When $ n = 0 $, (\ref{eq:resolution}) is \emph{a length $ 0 $ resolution of $ X_{-1} $} if $ f_0 $ is an equivalence $ X_0 \simeq X_{-1} $.
		\item When $ n> 0 $, (\ref{eq:resolution}) is \emph{a length $ n $ resolution of $ X_{-1} $} if 
		\begin{equation*} 
			\cofib(X_n \to X_{n-1}) \xrightarrow{\overline{f_{n-1}}} X_{n-2} \xrightarrow{f_{n-1}} \cdots X_{0} \xrightarrow{f_0} X_{-1} 
		\end{equation*} 
		is \emph{a length $ n -1 $ resolution of $ X_{-1} $}, where the homotopy class of $ \overline{f_{n-1}} $ is determined by the condition that $ f_{n-1} \circ f_n \simeq 0 $. 
	\end{enumerate}
	
	\item Let a length $ n $ resolution of $ X_{-1} $ (\ref{eq:resolution}) be given. 
	Writing $ Z_0 = X_0 $ and $ Z_n = X_n $ and $ Z_{i} \simeq \cofib(Z_{i+1} \to X_i) $, the data of a resolution is equivalently the data of $ n $ cofiber sequences: 
	\begin{equation}\label{eq:exactsequence_gluingdata}
		 Z_{i+1} \xrightarrow{h_i} X_i \xrightarrow{g_i} Z_i . 
	\end{equation} 
	We define the associated morphism \emph{classifying the resolution (\ref{eq:resolution})} to be the homotopy class of the composite 
	\begin{equation}
		X_0 \to \cofib(f_{0}) \simeq \Sigma Z_{n-2} \to \Sigma^2 Z_{n-3} \cdots \to \Sigma^{n-2}Z_1 \to \Sigma^{n-1} X_0 
	\end{equation}
	where each composite is a shift of the canonical map $ Z_i \to \cofib(g_i) \simeq \Sigma \fib(g_i) \simeq \Sigma Z_{i-1} $. 
\end{enumerate}
\end{defn}
\begin{rmk}
	Though it may seem unnatural to ask for nullhomotopies in an $ \infty $-category as a \emph{property} or to avoid defining resolutions in a suitably functorial way, our main purpose for this definition is for explicit bounds (cf. Lemma \ref{lemma:resolution_subadditive} and its application in Proposition \ref{prop:complexitybootstrap}). 
\end{rmk}
\begin{exs}
\begin{itemize}
	\item A resolution of length $ 1 $ is a cofiber sequence $ X_1 \to X_0 \to X_{-1} $. 
	\item A resolution of length $ 2 $ is a sequence of morphisms $ X_2 \xrightarrow{f_2} X_1 \xrightarrow{f_1} X_0 \xrightarrow{f_0} X_{-1} $ such that each composable pair is nullhomotopic and there exist nullhomotopies inducing an equivalence $ \cofib(f_2) \simeq \fib(f_0) $. 
	\item Let $ \mathcal{A} $ be a nice abelian category. 
	Let $ P_n \to P_{n-1} \to \cdots \to P_0 \to M $ be a projective resolution of $ M \in \mathcal{A} $. 
	Then it is a resolution in the sense of Definition \ref{defn:resolution} in the derived category $ D(\mathcal{A}) $ \cite[Definition 1.3.5.8]{HA}.
\end{itemize}
\end{exs}
We record an observation which will be used in \S\ref{section:twist_entropy_results}. 
\begin{lemma}\label{lemma:resolution_subadditive}
	Let $ \cat $ be a stable $ \infty $-category and let $ Y \in \cat $.
	Suppose given an exact sequence (\ref{eq:resolution}) in $ \cat $. Then
	\begin{equation}\label{eq:resolution_subadditive}
		\delta(Y, X_n) \leq \sum_{i=1}^n \delta\left(Y, \Sigma^{i-1} X_{n-i}\right) .
	\end{equation} 
\end{lemma}
\begin{proof}
	This follows from consideration of Lemma \ref{lemma:complexityprops}(\ref{lemmaitem:complexity_subadditive}) applied to the exact sequences of (\ref{eq:exactsequence_gluingdata}).
\end{proof}

\begin{defn}\label{defn:catentropy} \cites[Definition 2.4]{MR3289326}[Definition 2.4]{MR4233273}
	Let $ \cat $ be a stable $ \infty $-category with a generator $ G $, and suppose $ F: \cat \to \cat $ is an exact endofunctor. 
	Then the \emph{categorical entropy} of $ F $ is
	\begin{align*}
		h_\mathrm{cat}(F) &:= \lim_{n \to \infty} \frac{\log \delta(G, F^n(G))}{n}
	\end{align*}
	and the \emph{categorical polynomial entropy} of $ F $ is
	\begin{align*}
		h_{\pol}(F) &:= \limsup_{n \to \infty} \frac{\log \delta(G, F^n(G)) - n \cdot h_{cat, t}(F)}{\log n} .
	\end{align*}
\end{defn}
Note that our definition of $ h_{\mathrm{cat}}(F) $ (resp. $h\pol(F)$) agrees with the categorical entropy \cite[Definition 2.4]{MR3289326} (resp. categorical polynomial entropy \cite[Definition 2.4]{MR4233273}) functor $ F $ induces on the triangulated category $ \ho \cat $.   
\begin{rmk}\label{rmk:entropyfunctionorvalue}
	In the language of \cite{MR3289326,MR4233273}, we have defined the values of the categorical entropy function $ h_{\mathrm{cat}, -}(F): \R \to \R \cup \{-\infty\} $ and the categorical polynomial entropy function $ h_{\pol,-}(F) :\R \to \R \cup\{-\infty \} $ when $ t = 0 $. 
	These values might be interpreted as `geometric' (Observation \ref{obs:topent_comparison}). 
\end{rmk}
We collect a few basic properties of categorical entropy below. 
\begin{prop} \cite[Lemma 2.5]{MR3289326}
	Let $ \cat $ be a small stable $ \infty $-category admitting a single generator, and $ F $ an exact endofunctor of $ \cat $. Then
	\begin{itemize}
		\item The categorical entropy $ h_\mathrm{cat}(F)$ and the categorical polynomial entropy $ h_{\pol}(F) $ of $ F $ are well-defined, i.e. their defining limits exist in $ \R \cup\{-\infty\} $. 

		\item The categorical entropy $ h_\mathrm{cat}(F)$ and the categorical polynomial entropy $ h_{\pol}(F) $ of $ F $ are independent of a choice of generator of $ \cat $. 

		\item The categorical entropy and the categorical polynomial entropy of the identity endofunctor are identically zero: $ h_\mathrm{cat}(\id_\cat) = h_{\pol}(\id_\cat) = 0 $.
	\end{itemize}
\end{prop}
\begin{rmk}
	Categorical entropy is meant to measure the complexity of a functor. 
	It makes sense that the entropy of the identity functor is zero. 
\end{rmk} 
\begin{obs}\label{obs:topent_comparison}
	A special case of categorical entropy can be related to topological entropy. 
	For instance, let $ {\cat = D^b(\Coh_X)} $ be the bounded derived category of coherent sheaves on a smooth projective variety $ X $ over the complex numbers. 
	A self-map $ f:X \to X $ of $ X $ induces a pullback functor $ F = f^*: \cat \to \cat $. 
	Under a certain condition on $ f^* $, the categorical entropy of $ f^* $ is bounded from below by the logarithm of the spectral radius of the induced linear map $ f^* $ on Kähler differentials $ \Omega^n_{X/\mathbb{C}} $ by \cite[Theorem 2.8]{MR3289326} and the Hochschild--Kostant--Rosenberg theorem \cite{MR142598}. 
	When $ f $ is surjective, the formality of Kähler manifolds \cite[Main Theorem]{MR382702} and the Gromov-Yomdin theorem \cites[2.3 Corollaire]{MR2026895}[Theorem 1.1]{MR889979} imply that the latter is equal to the topological entropy of $ f $. 
\end{obs}

\subsection{Linear-algebraic bounds.} \label{subsection:linalg_bounds_ent}
A classical theorem of Yomdin \cite{MR889979} provides a lower bound for the entropy of a $ C^\infty $-smooth self-map of a compact real manifold $ M $ in terms of the induced map on cohomology $ H^*(M ; \R) $.  

Homological bounds for topological entropy are obtained under certain niceness conditions on the underlying space and map. 
The relevant niceness condition(s) for stable $ \infty $-categories are smoothness and properness, which together are known as dualizability. 
We characterize these conditions following the exposition of \cite[\S3.2]{MR3190610}.

\begin{recollection} 
	[$ k $-linear stable $ \infty $-categories]
	Let $ k $ be a commutative ring. 
	Let $ \Pr^L $ be the category of presentable stable $ \infty $-categories and left exact functors \cite[Definition 5.5.3.1]{LurHTT}. 
	Then $ \Mod_k $ is an $ \E_\infty $-algebra object in $ \Pr^L $. 
	Write $ \Cat_k = \Mod_{\Mod_k}(\Pr^L) $ for the category of $ k $-linear presentable stable $ \infty $-categories, and $ \Cat^\omega_k $ for the category of compactly-generated $ k $-linear categories and colimit-preserving functors that preserve compact objects. 
	The category $ \Cat_k $ inherits a symmetric monoidal structure from $ \Pr^L $ by \cite[Proposition 4.5.3.1]{HA}, which we denote by $ \otimes_k $. 

	Given $ \cat \in \Cat_k $, its dual is the functor category
	\begin{equation*}
		\cat^\vee = \Fun^L_k(\cat, \Mod_k)
	\end{equation*} 
	in $ \Mod_k $. 
	There is a functorial evaluation map
	\begin{equation*}
		\cat \otimes_k \cat^\vee \to \Mod_k .
	\end{equation*}
\end{recollection}
\begin{defn}\label{def:smoothpropcat} \cite[Definition 3.8]{MR3190610}
	The $ k $-linear category $ \cat \in \Cat_k $ is \emph{dualizable} if there exists a coevaluation map 
	\begin{equation*}
		\coev: \Mod_k \to \cat^\vee \otimes_k \cat
	\end{equation*}
	which classifies $ \cat $ as a $ \cat^\vee \otimes_k \cat $ and such that both composites
	\begin{align*}
		\cat \xrightarrow{\id_{\cat} \otimes \coev} \cat \otimes_k \cat^\vee \otimes_k \cat \xrightarrow{\ev \otimes \id_{\cat}} \\
		\cat^\vee \xrightarrow{\coev \otimes \id_\cat} \cat^\vee \otimes_k \cat \otimes_k\cat^\vee \xrightarrow{\id_\cat \otimes_k \ev} \cat^\vee
	\end{align*}
	are equivalent to the identity. 

	A compactly-generated $ k $-linear category $ \cat \in \Cat^\omega_k $ is \emph{proper} if its evaluation map preserves compact objects; it is \emph{smooth} if it is dualizable and its coevaluation map is in $ \Cat^\omega_k $. 

	Finally, a small stable $ k $-linear category $ \mathcal{D} $ is proper (resp. smooth) if its Ind-completion $ \Ind(\mathcal{D}) $ is.
\end{defn} 
\begin{ex}
	Let $ \cat = \Mod_A $ for some $ k $-algebra $ A $. 
	Then by the discussion following \cite[Definition 3.8]{MR3190610},
	\begin{itemize}
		\item The category $ \cat $ is proper if and only if $ A $ is proper as a $ k $-module. 
		\item The category $ \cat $ is smooth if and only if $ A $ is perfect as an $ A \otimes_k A^\op $-module. 
	\end{itemize} 
\end{ex}
\begin{warning} 
	The literature calls a dg-algebra $ A $ smooth if its category of modules $ \Mod_A $ is.
	On the other hand, suppose $ k $ is a perfect field and consider a discrete commutative ring $ B $ as a dg-algebra $ B^c $ concentrated in degree zero.
	Then smoothness of $ B^c $ corresponds to \emph{regularity} of $ B $ in the sense of ordinary commutative algebra by \cite[Propositions 3.8 \& 3.13]{MR2609187}. 
\end{warning} 
\begin{prop}\label{prop:poincarepolybound} 
	Let $ k $ be a field, and suppose $ \cat $ is a $ k$-linear stable $ \infty $-category and $ G $ is a compact generator of $ \cat $. 
	If $ \cat $ is \emph{proper}, then 
	\begin{align*}
		h_{cat}(F) &\geq \lim_{N \to \infty}\frac{1}{N} \log \sum_\ell \dim_k \Ext^{-\ell}(G, F^N(G))e^{\ell t} \\
		h_{pol}(F) &\geq \limsup_{N \to \infty}\frac{1}{\log N}\left(\log \sum_\ell \dim_k \Ext^{-\ell}(G, F^N(G))e^{\ell t} - N h_{cat}(F) \right).
	\end{align*} 
	If $ \cat $ is smooth, then the reverse inequality holds. 
\end{prop}
\begin{proof}
	This follows from an examination of the proofs of \cite[Theorem 2.6]{MR3289326} and \cite[Lemma 2.7]{MR4233273}.
\end{proof}

\section{Graded Hopf algebras}\label{section:hopfalgebras} 
A cohomology theory $ E:\Spc^\op \to \AbGrp $ is an invariant of spaces.
Any $ E $ admits \emph{cohomology operations}--in particular they may be assembled into a graded Hopf algebra $ E^*E $ \cites[Theorem 6.10]{MR22071}[p.215]{MR54965}. 
Considering $ E $ as valued in modules over its algebra of operations provides a more refined invariant lifting $ E $. 
In particular, two spaces $ X ,Y $ which are not homotopy equivalent may have abstractly isomorphic $ E $-cohomology \emph{groups} but lift to different $ E^*E $-module structures. 
We can do even better--the Adams spectral sequence \cite[Theorem 2.1]{MR96219} has $ E_2 $-page the cohomology of $ E^*E $ and converges to (a localization of) the stable homotopy groups of spheres. 

In this section we provide a terse summary of the requisite background on graded Hopf algebras and their module categories. 
The interested reader should refer to \cite{MR174052,Swe69} for more context and history.
The reader who is not interested in absurd generalities should feel free to consider one of Examples \ref{ex:liegp_homology} or \ref{ex:mod2steenrod} in place of $ A $ throughout. 
After introducing definitions and establishing conventions in \S\ref{subsection:hopf_alg_defn}, we describe the analogue of short exact sequences for Hopf algebras in \S\ref{subsection:hopf_alg_ext} and recall a result of Wilkerson guaranteeing the existence of extensions of a certain form. 
We collect tools for computing Hopf algebra cohomology in \S\ref{subsection:hopf_cohomology} and introduce the stable module category in \S\ref{subsection:stablemodcats}. 
A less computationally-minded reader should feel free to black box Propositions \ref{prop:CE_sseq_coh} and \ref{prop:noncommex_cohkrulldim_computed} on a first read. 

\subsection{Definitions.}\label{subsection:hopf_alg_defn} 
Let $ \Z^\delta $ be the discrete category (i.e., no nonidentity morphisms) with objects given by the integers. 
The category  $ \Z^\delta $ has a symmetric monoidal structure given by addition of integers.
\begin{recollection}\label{rec:graded_vs}
Let $ k $ be a field, and let $ \Vect_k^{\gr} = \Mod_k^{\heartsuit, \mathrm{gr}} = \Fun(\Z^\delta, \Vect_k) $ be the category of $ \mathbb{Z} $-graded $ k $-vector spaces. 
\begin{enumerate}[label=(\alph*)]
	\item Given a graded $ k $-vector space $ A $, we will refer to its image under the forgetful functor 
	$	{u: \Vect_k^{\gr} \to \Vect_k} $, $ u: V \mapsto \bigoplus_{\ell \in \Z} V_\ell $ as its \emph{underlying} $ k $-module. 

	\item A graded $ k $-vector space $ A $ is said to be \emph{finite-dimensional} or \emph{finite} if its underlying $ k $-module is finite-dimensional, and \emph{levelwise finite-dimensional} if $ A_\ell $ is finite-dimensional over $ k $ for all $ \ell $. 

	Given a levelwise finite-dimensional graded $ k $-vector space $ A $, its \emph{Hilbert} or \emph{Poincaré series} is the formal power series $ p_A(t) := \sum_{\ell} (\dim_k A_\ell) t^\ell $. 
	\item \label{recitem:twist_functor_vs} Let $ V $ be an ungraded $ k $-vector space. 
	For any integer $ n $, we write $ V(n) $ for the graded $ k $-vector space which is $ V $ in grading $ n $ and zero otherwise. 

	For every integer $ a \in \Z $, there is an automorphism $ (a) $ on graded $ k $-vector spaces given by precomposing with the automorphism $ \Z \xrightarrow{- a} \Z $, i.e. this is given by $ M(a)_\ell = M_{\ell - a} $. 
	\item Since $ \Vect_k $ has a monoidal structure given by $ \otimes_k $, $ \Vect_k^{\gr} $ inherits a monoidal structure via Day convolution. 
	On objects, the tensor product in $ \Vect_k^\gr $ of $ A, B \in \Vect_k^\gr $ is given by 
	\begin{equation*} 
		(A \otimes B)_\ell = \bigoplus_{m} A_m \otimes_k B_{\ell-m} , 
	\end{equation*}
	with unit given by $ \mathbbm{1}_\ell = k $ if $ \ell = 0 $ and $ 0 $ otherwise. 
	The tensor product on $ \Vect_k^\gr $ is canonically symmetric monoidal via an isomorphism $ \tau: A \otimes B \simeq B \otimes A $ which obeys the \emph{Koszul sign rule}, i.e. which picks up a sign of $( -1)^{pq} $ on the component $ {A_p \otimes B_q} $. 
	\item We will denote the category of associative algebras in $ \Vect_k^\gr $ with respect to this symmetric monoidal structure by $ \Gr\Alg^\heartsuit_k $. 
	Observe that $ \Gr\Alg^\heartsuit_k $ inherits a symmetric monoidal structure $ \otimes $ from $ \Vect^\gr_k $. 

	\item \label{recitem:graded_enrichment}	We may regard $ \Vect_k^\mathrm{gr} $ and $ \Mod_k^\mathrm{gr} $ as being enriched over itself: we can consider 
	\begin{equation*} 
		\Hom_{\Mod_k^{gr}}(M, N)_{-\ell} := \Hom_{\Fun(\Z^\delta, \Vect_k)}(M, N(\ell)) 
	\end{equation*} 
	and likewise for $ \Mod_k^\mathrm{gr} $. 
	In particular, taking $ N = k(0) $, we have a duality functor $ {(-)^\vee: \Vect^{gr, \op}_k \xrightarrow{} \Vect^{gr}_k} $ such that $ {M^\vee_\ell = hom_k(M_{-\ell}, k)} $. 

	Given a graded $ k $-module spectrum $ M $, its homology groups are bigraded: $ \Ext^{s,t}(k,M) = \pi_s\hom_{\Fun(\Z^\delta, \Mod_k)} (k(t),M) $. 
	When a bigraded group arises as the bigraded homology groups of a graded $ k $-module spectrum, we refer to $ s $ as the (co)homological grading and $ t $ as the internal grading or weight.  
\end{enumerate}
\end{recollection} 
\begin{rmk}
	Taking duals restricts to an equivalence between levelwise finite-dimensional commutative (resp. associative) algebras in $ \Vect^\gr_k $ and levelwise finite-dimensional cocommutative (resp. coassociative) coalgebras in $ \Vect^\gr_k $.  
\end{rmk}
We say that a graded $ k $-vector space is \emph{connected} if $ A_\ell = 0 $ if $ \ell < 0 $ and $ A_0 = k $, and a graded $ k $-algebra is connected if its underlying $ k $-vector space is. 
\begin{warning}
	There is a clash of terminology here: A graded $ k $-vector space $ A $ is said to be {connected} if it is concentrated in nonnegative grading and $ A_0 \simeq k $ in the convention of \cite[p.26]{MR174052}.  
	This is to be contrasted with its other possible interpretation of being connective with respect to a t-structure on $ \Perf_k^\gr $. 
\end{warning}
\begin{defn}\label{defn:hopf_alg}
	Let $ k $ be a field and let $ A $ be a graded algebra in $ k $-vector spaces. 
	Suppose we are given a collection of morphisms in $ \Gr\Alg_k^\heartsuit $
	\begin{itemize}
		\item $ \Delta:A \to A \otimes_k A $, 
		\item $ \varepsilon: A \to k $, and 
		\item $ c: A \to A^\op $ where $ A^\op $ is the \emph{opposite algebra to $ A $} 
	\end{itemize}
	such that $ \Delta $, $ \varepsilon $, $ c $ form the comultiplication, counit, and coinverse of a cocommutative cogroup structure on $ A \in \Gr\Alg_k^\heartsuit $. 
	We say the data of $ (A, \Delta, \varepsilon, c) $ is a cocommutative, graded Hopf algebra over $ k $. 
	Often $ c $ is referred to as an \emph{antipode} or \emph{conjugation}. 

	A morphism of finite-dimensional, graded, cocommutative Hopf algebras over $ k $ is a morphism of graded $ k $-algebras $ \phi: A \to B $ which respects the cogroup structures on $ A $ and $ B $. 
	Given such a morphism $ \phi: A \to B $, $ A $ is said to be a \emph{sub-Hopf algebra} (resp. $ B $ is a \emph{quotient Hopf algebra} of $ A $) of $ B $ if the morphism $ \phi $ is injective (resp. surjective). 

	A graded Hopf algebra over $ k $ is said to be \emph{connected (resp. finite-dimensional)} if its underlying graded $ k $-vector space is. 

	We write $ I(A):= \ker(\varepsilon: A \to k) $ and call $ I(A) $ the \emph{augmentation ideal of $ A $}.
\end{defn}
\begin{obs}
	Spelling out the commutativity (resp. cocommutativity) axiom in more detail, this asks for the following diagrams 
	\begin{equation*}
	\begin{tikzcd}[column sep=tiny]
		A \otimes A \ar[rr,"\tau"] \ar[rd, "\mu"'] &&  A \otimes A \ar[ld,"\mu"] \\
		& A & 
	\end{tikzcd} \qquad \mathrm{and} \qquad
	\begin{tikzcd}[column sep=tiny]
		& A \ar[ld, "\Delta"'] \ar[rd,"\Delta"] & \\
		A \otimes A \ar[rr,"\tau"] && A \otimes A
	\end{tikzcd} 
	\end{equation*}
	to commute.
	By definition of $ \tau $, this agrees with what is often referred to as \emph{graded-(co)commutative}. 
\end{obs}
We will occasionally drop $ \Delta, \varepsilon, c $ and say $ A $ is a graded Hopf algebra over $ k $ if the Hopf algebra structure is understood. 
\begin{rmk}
	Suppose $ A $ is a connected graded bialgebra over a field $ k $ (that is, a graded Hopf algebra minus the assumption of an antipode). 
	Then an antipode for $ A $ exists automatically \cite[p. 185-6]{Margolis83}. 
	Thus we do not specify the antipode in the following examples. 
\end{rmk}
\begin{ex} [Trivial Hopf algebra] \label{ex:trivialHopfalgebra}
	The symmetric monoidal unit $ \mathbbm{1} = k(0) $ is canonically a finite connected graded bicommutative Hopf algebra over $ k $. 
	The object $ \mathbbm{1} $ is both initial and terminal in the category of finite-dimensional connected graded cocommutative Hopf algebras over $ k $. 
\end{ex}
\begin{ex}\label{ex:liegp_homology}
	Let $ G $ a compact, connected Lie group over $ \R $. 
	Then the Pontrjagin product on $ H_*(G; k) $, combined with the K\"unneth isomorphism, makes $ H_*(G;k) $ into an example of a finite connected graded algebra over $ k $. 
	The diagonal map $ \Delta: G \to G \times G $ induces a cocommutative coproduct on $ H_*(G; k) $ making $ H_*(G; k) $ into a connected graded cocommutative Hopf algebra over $ k $.  
\end{ex} 
\begin{obs}\label{obs:quotient_subHopfalgebra_duality}
	Let $ A $ be a Hopf algebra over $ k $. 
	Then taking $ k $-linear duals exhibits a bijection between sub-Hopf algebras of $ A $ and quotient Hopf algebras of the dual $ A^\vee $. 	
\end{obs}
\begin{ex}
	[Sub-Hopf algebras of the Steenrod algebra] \label{ex:mod2steenrod} Observation \ref{obs:quotient_subHopfalgebra_duality} has been employed to make the following identification.  
	We write $ \mathcal{A} $ for the mod 2 Steenrod algebra and $ \mathcal{A}^* $ for its dual. 
	The Hopf algebra $ \mathcal{A} $ is the algebra of cohomology operations $ H\F_2^*H\F_2 $ for singular cohomology theory with $ \F_2 $ coefficients. 
	Let 
	\begin{equation*}
	 E :=\left\{ e: \{1, 2, \ldots\} \to \Z_{\geq 0} \cup\{\infty\} \middle| e(i) \geq \min_{j<i}\{e(j), e(i-j)-j \} \text{ for all } i  \right\}.
	\end{equation*}
	Anderson and Davis showed \cite{MR334207} that the assignment $ e \mapsto \mathcal{A}^*/\left(\zeta_i^{e_i}\right) $ defines a bijection from $ E $ to the collection of sub-Hopf algebras of $ \mathcal{A} $. 
	The sub-Hopf algebra associated to an exponent sequence $ e $ is finite over $ \F_2 $ if and only if $ e(i) = 1 $ for all but finitely many $ i $. 

	We write $ \mathcal{A}_n $ for the sub-Hopf algebra of $ \mathcal{A} $ corresponding to the exponent sequence $ \{n+1, n, \ldots, 1,0, 0, \ldots \} $. 
	Then \cite[Chapter II]{MR2623793} has shown
	\begin{enumerate}
		\item The sub-Hopf algebra $ \mathcal{A}_n $ of $ \mathcal{A} $ is exactly the sub-Hopf algebra generated by $ \Sq^1, \ldots, \Sq^{2^n} $.  
		\item The $ \mathcal{A}_n $ exhaust the \emph{nice} sub-Hopf algebras of $ \mathcal{A} $; that is, these are the only sub-Hopf algebras of $ \mathcal{A} $ which admit an $ \mathcal{A} $-action which is compatible with the inclusion $ \mathcal{A}_n \subseteq \mathcal{A} $. 
		Here $ \mathcal{A}_n $ acts on itself by multiplication.  
	\end{enumerate}
While $ \mathcal{A} $ is not bicommutative, it is cocommmutative. 
Multiplication in the Steenrod algebra is governed by the Adem relations \cite[Theorem 1.1]{MR50278}: for all $ s > t $,
\begin{equation}\label{eq:ademrelations}
	\Sq^{2t} \Sq^s = \sum_{j=0}^t {s-t+j-1 \choose 2j}\Sq^{t+s+j}\Sq^{t-j} 
\end{equation} 
where the binomial coefficients are to be interpreted mod $ 2 $ and $ {k \choose j} =0 $ if $ j >  k $. 
\end{ex}

\subsection{Extensions of Hopf algebras.}\label{subsection:hopf_alg_ext}
A short exact sequence of groups $ {\{e\} \to N \mathrel{\unlhd} G \to G/N \to \{e\}} $ exhibits $ G $ as being built or glued from $ N $ and $ G/N $ in a twisted manner.
There is an analogous notion of a Hopf algebra $ C $ being a twisted product or extension of a quotient Hopf algebra by a normal sub-Hopf algebra. 

\begin{defn}\label{defn:hopf_alg_extension}
	Let $ A, B, C $ be connected graded Hopf algebras over $ k $. 
	A pair of morphisms of Hopf algebras $ i: A \to B $, $ p: B \to C $ is said to exhibit $ B $ as \emph{a Hopf algebra extension of $ C$ by $ A $} if there exists a map $ B \to A \otimes_k C $ which is both an isomorphism of left $ A $-modules and right $ C $-comodules. 
\end{defn} 

\begin{rmk}\label{rmk:hopf_alg_quotient_in_extension}
	In particular, given any such extension we have $ k \otimes_A B \simeq C $ by \cite[Proposition 4.9]{MR174052}.
\end{rmk}

We isolate a particularly nice class of extensions generalizing the notion of central extensions of groups. 
\begin{defn} \cite[\S6]{MR143788}
	A map of graded $ k $-algebras $ \varphi: A \to B $ in $ \Gr\Alg_k $ is said to be \emph{central} if the following diagram commutes
	\begin{equation*}
	\begin{tikzcd}[column sep=tiny]
		A \otimes B \ar[rr,"\tau"] \ar[d, "\varphi \otimes \id"'] & &B \otimes A \ar[d,"\id \otimes \varphi"] \\
		B \otimes B \ar[rd,"\mu_B"'] & & B \otimes B \ar[ld,"\mu_B"] \\
		& B &
	\end{tikzcd}.
	\end{equation*}
	There is a dual notion for maps of graded $ k $-coalgebras. 

	A Hopf algebra extension $ A \xrightarrow{i} B \xrightarrow{p} C $ is \emph{central} if $ i $ is a central map of algebras and $ p $ is a central map of coalgebras. 
\end{defn}
\begin{ex}\label{ex:SteenrodA1_extension}
	Let $ \mathcal{A}_1 $ be the sub-Hopf algebra of the Steenrod algebra generated by $ \Sq^1, \Sq^2 $ (see Example \ref{ex:mod2steenrod}). 
	Recall that $ \Sq^1 $ and $ \Sq^2 $ do not commute, and we write $ Q_1 = [\Sq^1, \Sq^2] $ for their commutator. 
	Then there is a central extension of cocommutative, finite-dimensional Hopf algebras 
	\begin{equation*} 
		\F_2 \to \F_2[Q_1]/(Q_1^2) \to \mathcal{A}_1 \to \bigwedge_{\F_2}[\Sq^1, \Sq^2] \to \F_2 
	\end{equation*}
	over $ \F_2 $. 
\end{ex}

Next we characterize a class of relatively simple commutative Hopf algebras, and recall a result of Wilkerson which shows these are `building blocks' for the more complicated Hopf algebras we will eventually consider. 
\begin{defn} 
	Let $ A $ be a Hopf algebra. 
	Given an element $ a \in A $, its \emph{height} is the minimal $ h \in \Z_{>0} $ such that $ a^h = 0 $. 	

	A Hopf algebra $ A $ over a field $ k $ of characteristic $ p $ is \emph{elementary} if all elements in the augmentation ideal have height $ p $, i.e. $ I^p $. 

	A Hopf algebra $ A $ is \emph{monogenic} if $ T $ can be taken to consist of a single element. 

	Given a graded Hopf algebra $ A $, a collection of elements $ T \subseteq A $ is said to \emph{generate} $ A $ if $ T $ generates $ A $ as an algebra. 
\end{defn} 
Recall that an element of a Hopf algebra $ a \in A $ is said to be \emph{primitive} if its image under the comultiplication is given by $ \Delta(a) = 1 \otimes a + a \otimes 1 $. 
\begin{exs}[Monogenic Hopf algebras of minimal height]\label{ex:monogenichopfalgs}
\cite{MR597872} Let $ k $ be a field of any characteristic. 
\begin{enumerate}
	\item There is a monogenic elementary Hopf algebra $ \bigwedge_k [x] $ where $ x $ is primitive, for $ x $ in any degree if char $ k = 2 $ or $ x $ in odd degrees when $ {\mathrm{char}\;  k = p} $ or $ 0 $.

	\item If char $ k $ is odd, then $ k[x]/(x^p) $ with $ x $ primitive in even grading is an elementary Hopf algebra.
\end{enumerate}
\end{exs} 
\begin{recollection}
	Borel's structure theorem \cite[Théorème 6.1]{MR51508} characterizes all bicommutative Hopf algebras over a perfect field as tensor products of monogenic Hopf algebras.  
\end{recollection}
In fact when $ \mathrm{char}\; k = 0 $ we have the following characterization.
\begin{prop}\label{prop:structure_hopfalg_charzero} \cite[Proposition 1.1]{MR597872}
	Let $ A $ be a finite-dimensional connected graded cocommutative Hopf algebra over a field $ k $ characteristic zero. 
	Then $ A $ is isomorphic as a Hopf algebra to a tensor product of exterior algebras. 
\end{prop}
The following result of Wilkerson can be thought of as a generalization of Borel's result to associative, cocommutative Hopf algebras. 
\begin{prop}\label{prop:existence_monogen_subhopf}
	\cite[Proposition 1.2]{MR597872} 
	Let $ A $ be a finite-dimensional connected graded cocommutative Hopf algebra over a field $ k $ of arbitrary characteristic. 
	Then there exists a nontrivial monogenic sub-Hopf algebra $ C $ of $ A $ of minimal height. 
	In particular $ C $ can be taken to be a monogenic Hopf algebra of the forms in Examples \ref{ex:monogenichopfalgs}.
\end{prop}
In particular, we can write $ A $ as an extension of $ C $ by $ k \otimes_C A =: B $. 

\subsection{Modules and base change.}\label{subsection:hopf_cohomology} 
Let $ A $ be a graded Hopf algebra over $k $. 
In this section we give both a model-independent description of the cohomology of $ A $, and also introduce explicit techniques and resolutions for computing the (bigraded) cohomology groups of $ A $. 
In particular, we show that when $ A $ is cocommutative, its cohomology is a graded-commutative ring, so the Krull dimension of $ H^*(A; k) $ is well-defined. 
The latter result is the only part of this section used in \S\ref{section:twist_entropy_results}; computational tools for understanding Hopf algebra cohomology will not be used until \S\ref{section:examples}. 

Recall our notation $ \Mod^\heartsuit_A $ for the abelian (1-)category of discrete, graded modules over $ A $.   
Since taking Eilenberg--Maclane spectra $ H(-): \Vect_k \to \Mod_{k}(\Spectra) $ is symmetric monoidal, we can equivalently regard $ A $ as a bialgebra in graded $ Hk $-module spectra. 
We will abuse notation by writing $ A $ for the image of $ A $ under the above map, and we write $ \Mod_A $ for the $ \infty $-category of modules in graded spectra over $ HA $. 
Write $\Mod_A\left(\Perf_k\right)$ for the $ \infty $-category of graded left $ A $-module spectra whose underlying spectrum is perfect over $ k $.

From now on, we suppress ``gr''--all modules are understood to be graded unless specified otherwise. 
\begin{rmk}
	[Functoriality]
	Given a map of graded Hopf algebras $ \varphi: A \to B $, we have corresponding restriction 
	\begin{equation*}
	\varphi^*: \Mod_B \to \Mod_A  \qquad \varphi^* : \Mod^\heartsuit_B \to \Mod^\heartsuit_A 
	\end{equation*} 
	and induction functors 
	\begin{equation*}
	 B\otimes_A -:\Mod_A \to \Mod_B \qquad  B \otimes^\heartsuit_A -: \Mod^\heartsuit_A \to \Mod^\heartsuit_B .
	\end{equation*} 
	If $ B \otimes_A k $ is a perfect $ k $-module, then the induction functor further restricts to functor $ \Mod_A\left(\Perf_k\right) \to \Mod_B\left(\Perf_k\right) $. 
	In general the forgetful functor $ \Mod_B \to \Mod_A $ descends to a functor $ \Perf_B \to \Perf_A $ if and only if $ B $ is a perfect $ A $-module. 
\end{rmk}
A special case of the previous is the 
\begin{defn}\label{defn:hopf_alg_(co)homology}
Given a graded Hopf algebra $A $ over a field $k$, restriction along the counit $ \varepsilon: A \to k $ gives a functor $ \Mod_k^\gr \to \Mod^\gr_A $. 
The restriction functor admits both left and right adjoints, given respectively by
\begin{align*}
	& \Mod_A \to \Mod_k \\
	(-)_{A} &: M \mapsto M \otimes_A k \\
	(-)^A & :M \mapsto \hom_A(k, M)
\end{align*}
Again regarding $ k $ as an $ A $-module via the counit, the image of $ k $ under the left, right adjoints are denoted by $ k_A = C_*(A; k) $ the \emph{homology of $ A $} and $ k^A = C^*(A; k) $ the \emph{cohomology of $ A $}. 
\end{defn}
We will use freely the following characterization.
\begin{prop} \label{prop:exttor_modulespectra}
	Let $ k $ be a field and let $ A $ be a $ \Z $-graded associative $ k $-algebra which is levelwise discrete. 
	Then
	\begin{align*}
		\pi_{s, t} k \otimes_A k \simeq \Tor^A_{s,t}(k,k)  \\
		\pi_{-s, -t} \hom_A(k, k) \simeq \Ext^{s, t}_A(k, k)
	\end{align*}
\end{prop}
\begin{proof}
	The first equivalence is a straightforward generalization of \cite[Corollary 7.2.1.22]{HA} to the graded case. 
	The second equivalence is a straightforward generalization of \cite[Remark 7.1.1.16]{HA} to the graded case. 
\end{proof}
\begin{rmk}\label{rmk:composition_prod_cohomology}
	Note that composition on $ \hom_A(k,k) $ makes the bigraded homotopy groups $ \Ext^{s,t}_A(k,k) $ into an (a priori) graded, associative $ k $-algebra. 
	This is often referred to as the \emph{Yoneda product}. 
\end{rmk}
\begin{warning}\label{warning:homology_groupvsspectrum}
	We abuse notation and use homology (resp. cohomology) of $ A $ to refer to both the graded spectra $ k \otimes_A k $ (resp. $ \hom_A(k, k) $) \emph{and} their (bigraded) \emph{homotopy groups}. 
	For instance, the \emph{Krull dimension} of the cohomology of $ A $ refers to the Krull dimension of the graded ring $ H^*(A; k) $. 
	We make note of the grading convention $ \pi_{s,t} \hom_A(k,k) = \Ext^{-s,-t}_A(k,k) = H^{-s,-t}(A;k) $. 
	We write $ H^s(A; k) $ for the graded abelian group $ H^{s, *}(A; k) $. 

	We warn the reader unfamiliar with spectra and stable homotopy that for a general spectrum $ X $, $ \pi_*X $ contains less information than $ X $. 
\end{warning}
While the aforementioned descriptions of Hopf algebra cohomology are elegant, to compute the Krull dimension of $ H^*(A; k) $ in specific cases we will want to use more explicit descriptions.
\begin{cons}
	[Cobar construction] \cites[Definition A.2.11]{Ravenel:2003}[p.32-33]{MR141119} \label{cons:cbar_aug_filtration}
	Let $ A $ be a graded biassociative Hopf algebra over $ k $, and let $ M $ (resp. $ N $) be a left (resp. right) comodule over $ A $. 
	We write $ \overline{A} := \coker(\eta: k \to A ) $ for the cokernel of the unit, or the dual of the augmentation ideal of $ A^\vee $ (Definition \ref{defn:hopf_alg}).
	Define the \emph{cobar complex} $ \coBar_A(N; M) $ to be the cosimplicial graded abelian group
	\begin{equation*}
		\coBar_A(N; M)_s := N \otimes_k \overline{A}^{\otimes s} \otimes_k M 
	\end{equation*}
	with coboundary $ d^s: \coBar_A(N; M)_s \to \coBar_A(N; M)_{s+1} $ given by
	\begin{multline*}
		d^s (n \otimes \gamma_1 \otimes \cdots \otimes \gamma_s \otimes m) = \varphi(n) \otimes \gamma_1 \cdots \otimes \gamma_s \otimes \psi(m) +
		\sum_{i=1}^s (-1)^i n \otimes \cdots \otimes \Delta(\gamma_i) \otimes \cdots \otimes \gamma_s \otimes m \\
		+ (-1)^{s+1} n \otimes \cdots \otimes \gamma_s \otimes \psi(m) 
	\end{multline*}
	where $ \psi $ and $ \varphi$ are the comodule structure maps. 
	We abbreviate $ n \otimes \gamma_1 \cdots \otimes \gamma_s \otimes m = n\gamma_1|\gamma_2|\cdots|\gamma_sm $ and give it the usual grading, i.e. the grading degree of $ n\gamma_1|\gamma_2|\cdots|\gamma_sm $ is $ \sum_{i=1}^s|\gamma_i| + |n| + |m| $. 

	When $ M = N = k $, we may define a bilinear ``juxtaposition product'' \cite[discussion after Corollary 9.6]{MR820463} 
	\begin{align*}
		&\coBar_A(k;k)_p \otimes \coBar_A(k;k)_q \to \coBar_A(k;k)_{p+q} \\
		&( \gamma_1|\gamma_2|\cdots|\gamma_p) \otimes (\eta_1 | \cdots | \eta_q) \mapsto \gamma_1|\gamma_2|\cdots|\gamma_p | \eta_1 | \cdots | \eta_q .
	\end{align*}
\end{cons} 
The following is \cite[Corollary A.2.12]{Ravenel:2003}. 
\begin{prop}\label{prop:cobar_ext}
	Let $ k $ be a field, and let $ A $ be a graded Hopf algebra over $ k $. 
	Then the cohomology of the cosimplicial graded abelian group $ \coBar_{A^\vee}(k;k) $ computes
	\begin{equation*}
		H^s\coBar_{A^\vee}(k;k)_t = \Ext^{s,t}_A(k, k). 
	\end{equation*}
\end{prop}
\begin{prop}\label{prop:hopf_coh_gradedcomm}
	\cite[Theorems 9.7 \& 9.8]{MR1793722}
	Let $ A $ be a graded Hopf algebra over $ k $. 
	\begin{enumerate}
		\item The composition product of Remark \ref{rmk:composition_prod_cohomology} and the juxtaposition product of Construction \ref{cons:cbar_aug_filtration} coincide on bigraded homotopy groups. 
		\item If the coproduct on $ A $ is cocommutative, then the multiplication on $ \Ext^{s,t}_{A}(k,k) $ is graded-commutative. 
	\end{enumerate}
\end{prop}
The cobar complex is highly inefficient; a characterization of efficient resolutions is the 
\begin{defn}
	\cite[Definition 9.3]{MR1793722} \label{defn:minres}
	Let $ A $ be a discrete augmented ring over $ k $ and write $ \varepsilon: A \to k $ for the augmentation and $ I = \ker(\varepsilon) $. 
	A homomorphism of discrete left $ A $-modules $ f:M \to N $ is \emph{minimal} if $ f(M) \subseteq I \cdot N $. 
	A projective resolution of a module $ M $ is \emph{minimal} if every homomorphism in the resolution is minimal. 
\end{defn}
\begin{prop}
	\cite[Proposition 9.4]{MR1793722} \label{prop:minres_ext}
	Let $ A $ be a discrete augmented ring over $ k $ and write $ \varepsilon: A \to k $ for the augmentation and $ I = \ker(\varepsilon) $. 
	Let $ M $ be a left $ A $-module and suppose given $ \cdots P_2 \longrightarrow P_1 \longrightarrow P_0 \longrightarrow M $ a minimal resolution of $ M $ by projective left $ A $-modules. 
	Then $ \Ext^s_A(M, k) \simeq \hom_A(P_s, k) $.
\end{prop}
\begin{ex}\label{ex:exteriorhopf_cohomology_computed}
	Let $ A = \bigwedge_k(e_1, \ldots, e_n ) $ be an exterior algebra where each $ e_i $ is primitive. 
	Then $ \Ext_A^{s,t}(k, k) \simeq k[x_1, \ldots, x_n ] $ where $ |x_i| = (1, |e_i|) $. 

	While this computation may be well-known, we include it for completeness. 
	For each $ A_i = \bigwedge_k(e_i) $, there is a minimal periodic projective resolution of $ k $ by $ A_i $-modules
	\begin{equation*}
		 \cdots A_i(2|e_i|) \xrightarrow{\cdot x_i} A_i(|e_i|) \xrightarrow{\cdot x_i} A_i \to k.
	\end{equation*} 
	Tensoring these resolutions together over $ k $ and taking the diagonal gives a minimal resolution $ k \simeq \lim M^\bullet $ of $ k $ by projective $ A $-modules. 
	The computation follows from Proposition \ref{prop:minres_ext}. 
\end{ex}
Furthermore, the cocommutativity of $ A $ allows us to identify Steenrod operations on Hopf algebra cohomology \cites[Theorem 11.8]{MR0281196}[\S{II}.5]{MR182001}. 
\begin{prop}
	[Steenrod operations on Hopf algebra cohomology]  \label{prop:steenrod_hopf_coh}
	Let $ A $ be a graded, connected, cocommutative Hopf algebra over $ \F_2 $. 
	Then there are operations $ \Sq^i $ on $ \Ext_A^{*,*}(\F_2,\F_2) $ satisfying
	\begin{enumerate}
		\item The operation $ \Sq^i $ acts via $ \Sq^i: \Ext_A^{s,t}(\F_2,\F_2) \to \Ext_A^{s+i,2t}(\F_2,\F_2) $.
		\item (Cartan formula) $ \Sq^i(ab) = \sum_{m+n=i} \Sq^m(a)\Sq^n(b) $.
		\item (Adem relations) $ \Sq^r \Sq^s = \sum_{t=0}^{\floor{\sfrac{s}{2}}} \Sq^{r+s-t}\Sq^t $. 
		\item The operation $ \Sq^0 $ acts on cocycles by $ \Sq^0( \gamma_1|\gamma_2|\cdots|\gamma_n) = \gamma_1^2|\gamma_2^2|\cdots|\gamma_n^2 $.
		\item $ \Sq^s(x) = x^2 $ if $ x \in \Ext_A^{s,t}(\F_2,\F_2) $.
	\end{enumerate}
\end{prop}
\begin{warning}
	The reader who is used to Steenrod operations on the cohomology of spaces is cautioned that here $ \Sq^0 $ is not the identity. 
\end{warning} 
\begin{rmk}
	[On notation] 
	Wilkerson \cite{MR597872} writes $ \widetilde{\Sq^i} $ for the operations of Proposition \ref{prop:steenrod_hopf_coh} to remove ambiguity, as we occasionally will want to write $ \Sq^i $ for an element in $ A $ (Example \ref{ex:mod2steenrod}). 
	As this issue does not arise for us, we write $ \Sq^i $ to lighten notational burden. 
\end{rmk}
To compute the cohomology of a not necessarily commutative Hopf algebra is nontrivial. 
If our Hopf algebra $ A $ sits in an extension $ B \to A \to C $, we may attempt to understand the cohomology of $ A $ in terms of the cohomologies of $ B $ and of $ C $. 
The following proposition is a special case of the Cartan-Eilenberg spectral sequence.
\begin{prop}[Cartan-Eilenberg change-of-rings spectral sequence] \label{prop:CE_sseq_coh} 
	Let $ k $ be a field, and let $ A, B, C $ be graded Hopf algebras over $ k $. 
	Suppose given a pair of morphisms of cocommutative Hopf algebras $ i: B \to A $, $ p: A \to C $ exhibiting $ A $ as a Hopf algebra extension of $ C $ by $ B $. 
	\begin{enumerate}
		\item \label{propitem:CE_sseq_existence} There is a spectral sequence of trigraded algebras with $ E_2 $-page
		\begin{align*}
			E_2^{p,q} = \Ext_B^{p}(k,k) \otimes \Ext^q_C(k,k) \implies \Ext_A^{p+q}(k,k) .
		\end{align*}
		The differentials $ d_r \colon E_r^{p,q} \to E_r^{p+r, q-r+1} $ preserve internal degree. 
		\item \label{propitem:CE_sseq_leibniz} The differentials $ d_r $ are derivations, i.e. they satisfy the Leibniz rule:
		\begin{equation*}
			d_r(xy) = d_r(x)y + (-1)^{p+q}xd_r(y). 
		\end{equation*}
		\item \label{propitem:CEsseq_kudo} (Kudo transgression theorem) The Steenrod operations on $ E_r^{p,0} \simeq  \Ext_B^{p}(k,k) $ interact with the Steenrod operations on $ E_r^{0, q} \simeq \Ext^q_C(k,k) $ in the following way: If $ x \in E_2^{0,q} $ is such that $ d_i(x) = 0 $ for $ i < r $ and $ d_r(x) = y $, then $ d_{r+i}(\theta x) = \theta y $ where $ \theta $ is a Steenrod operation of homological degree $ i $. 
	\end{enumerate}
\end{prop}
\begin{proof}
\begin{enumerate}
	\item The additive construction is \cite[XVI.5, Case 4]{MR1731415}. 
	By \cite[Theorem A.1.3.16]{Ravenel:2003}, this spectral sequence coincides with that of \cite[Theorem 2.3.1]{MR141119}, where it is shown to be a spectral sequence \emph{of algebras}. 
	\item This follows from the definition of the spectral sequence and \cite[Theorem 2.14]{MR1793722}. 
	\item This is due to \cite[Theorem 3.4]{MR0281196}. \qedhere
\end{enumerate}
\end{proof}
\begin{rmk}
	The Steenrod algebra can be made to act on the entire spectral sequence and satisfies a more general version of the Kudo transgression theorem \cites[Propositions 1 \& 4]{MR357553}[Theorem 5.2]{MR2245560}. 
\end{rmk}
\begin{obs} [Künneth theorem] \label{obs:Kunneththm}
	Given the hypotheses of Proposition \ref{prop:CE_sseq_coh} and suppose that $ A \simeq B \otimes_k C $ as Hopf algebras. 
	Then the spectral sequence degenerates at $ E_2 $ and we have a Künneth theorem for Hopf algebra cohomology. 
\end{obs}
Using this spectral sequence and Steenrod operations on it, Wilkerson proved the following general result.   
\begin{theorem}
	\cite[Theorem A]{MR597872} \label{thm:ext_fingen}
	If $ A $ is a finite-dimensional graded, connected, cocommutative Hopf algebra over a field $ k $, then $ H^*(A; k) = \pi_{-*}\hom_A(k, k) $ is a finitely-generated $ k $-algebra. 
\end{theorem}
\begin{obs}\label{obs:homogenous_generators_bidegree}
	In fact, Wilkerson's proof shows that homogeneous polynomial generators $ x_i $ of $ H^*(A;k) = H^{s,t}(A;k) $ can be taken with bidegree $ |x_i| = (s_i, t_i) $ with both $ s_i, t_i $ strictly positive. 
\end{obs}
\begin{ex} \cite[Corollary G(ii)]{MR597872} \label{ex:SteenrodA1_cohom_Krulldim}
	Let $ \mathcal{A}_1 $ be the sub-Hopf algebra of the mod 2 Steenrod algebra generated by $ \Sq^1 $ and $ \Sq^2 $. 
	Then $ H^*(\mathcal{A}_1; \F_2) = \pi_* \hom_{\mathcal{A}_1}(\F_2, \F_2) $ has Krull dimension 2. 
\end{ex}

\subsection{Monoidal structures and duality.}\label{subsection:frobgor_duality}
While the extra structure on a graded Hopf algebra has not been used to define its category of modules, the Hopf algebra structure on $ A $ endows $ \Mod_A $ and $ \Mod_A^\heartsuit $ with extra structure as reflected below.  

\begin{cons} [The monoidal structure on $ \Mod_A $] \label{cons:tensor_Hopfmod}
The categories $ \Mod^\heartsuit_A, \Mod_A $ admit monoidal structures given by taking the image of $ A $ (considered as a coalgebra) under the functor 
\begin{align*} 
		\E_1\Alg(\Spectra)^\op &\to \Cat_\infty \\
		R &\mapsto \Mod_R  \\
		(\varphi:R \to S) &\mapsto  (\varphi^* : \Mod_S \to \Mod_R).
\end{align*} 
Informally, the underlying graded $ k $-module of $ M \otimes N $ is given by the tensor product $ M \otimes_k N $ in graded $ k $-modules, and the $ A $-module structure on $ M \otimes N $ is given by 
\begin{equation*} 
	A \otimes_k (M \otimes_k N) \xrightarrow{\Delta \otimes \id} A \otimes_k A \otimes_k (M\otimes_k N) \xrightarrow{\id \otimes \tau \otimes \id} (A \otimes_k M) \otimes_k (A \otimes_k N) \to M \otimes_k N .
\end{equation*} 
When $ A $ is cocommutative, this is moreover a symmetric monoidal structure. 
\end{cons}
\begin{warning}
This symmetric monoidal structure should not be mistaken for the standard monoidal structure on $ \Mod_A $ (for $ A $ a commutative algebra) given on objects by $ (M,N) \mapsto M\otimes_A N $.  
\end{warning}
\begin{obs}\label{obs:modules_unit_invertible}
	Endowed with the tensor product $ \otimes_k $, $ \Perf_A $ is not in general unital. 
	However the larger categories $ \Mod_A\left(\Perf_k\right) $ and $ \Mod_A $ are unital with unit given by $ k(0) $ with (co)module structure given by the (co)unit map. 

	Given a graded augmented $ k $-algebra $ A $, $ k(n) $ is always invertible with respect to the aforementioned symmetric monoidal structure with $ \otimes $-inverse given by $ k(-n) $. 
\end{obs}

\begin{defn}
	[{\cite[\S 12.2]{Margolis83}}] \label{defn:gorensteinalg}
	A finite-dimensional graded $ k $-algebra $ A $ is \emph{Gorenstein} or \emph{Poincaré} if there exists an integer $ n \in \Z $ and map of graded $ k $-modules $ e: A \to k(d) $ such that for each $ q \in \Z $, the pairing 
	\begin{equation*} 
		A_{d-q} \otimes A_q \xrightarrow{m} A_d \xrightarrow{e} k 
	\end{equation*} 
	is non-degenerate. 
	A choice of such an $ e $ is called an \emph{orientation}.
\end{defn}
By our assumption that $ A $ is finite-dimensional, if $ d $ exists it is unique, and we refer to it as the \emph{Gorenstein parameter}. 
Moreover, a choice of orientation for $ A $ is unique up to a choice of unit in $ k $. 

\begin{obs}
	Given a graded Hopf algebra $ A $ over $ k $, $ A^\vee $ acquires a canonical left $ A $-module structure on $ A^\vee $. 
	Informally, this is given by $ (a, f:A \to k) \mapsto f(- \cdot c(a)) $. 
\end{obs}
\begin{defn} \label{defn:frobeniusalg}
	A graded $ k $-algebra $ A $ is \emph{Frobenius} if there is an isomorphism of graded left $ A $-modules $ A \simeq A^\vee (d) \simeq \hom^\gr_k(A, k(d)) $. 
	A choice of such an isomorphism is an \emph{orientation} for $ A $. 
\end{defn} 
Proposition \ref{prop:frob_algebras_cats} shows that the notions of orientation of Definitions \ref{defn:frobeniusalg} and \ref{defn:gorensteinalg} agree. 
\begin{defn}
	An abelian category $ \mathcal{A} $ is \emph{Frobenius} if the following conditions are satisfied: 
	\begin{enumerate}
		\item The category $ \mathcal{A} $ has enough injectives.
		\item The category $ \mathcal{A} $ has enough projectives.
		\item The injectives and projectives in $ \mathcal{A} $ coincide. 
	\end{enumerate} 
\end{defn} 
\begin{prop}\label{prop:frob_algebras_cats}
	\cite[12.2.5]{Margolis83}
	Given a finite-dimensional connected graded $ k $-algebra $ A $, the following are equivalent:
	\begin{enumerate}
		\item The category of left modules over $ A $ $ \LMod_A^\heartsuit $ is a Frobenius category
		\item The algebra $ A $ is a Frobenius algebra
		\item The algebra $ A $ is a Poincaré algebra.
	\end{enumerate}
\end{prop}
\begin{prop}\label{prop:finconnhopfimpliesfrobenius}
	\cite[12.2.9]{Margolis83} If $ A $ is a finite, connected, graded Hopf algebra over a field $ k $, then $ A $ is a Frobenius algebra. 
\end{prop}
Given a Frobenius algebra $ A $ over a ground field $ k $ and an $ A $-module $ M $, we may define its \emph{Tate cohomology} via complete resolutions \cites[5.6.2]{MR4390795}[12.2]{MR1731415}. 
We give a different presentation for the same object in the derived category of $ k $.  
\begin{cons}\label{cons:tate}
	\cite[Construction 2.4.4]{Raksit20}
	Let $ A $ be a finite-dimensional, graded, cocommutative Hopf algebra over $ k $. 
	Then by Proposition \ref{prop:finconnhopfimpliesfrobenius}, there is some $ d > 0 $ and an equivalence of $ A $-modules $ A \xrightarrow{\sim} A^\vee \otimes_k k(d) $. 
	Let $ \varepsilon^\vee: k(0) \to A^\vee $ denote the dual of the counit. 
	Then given any $ A $-module $ M $, we have a morphism
	\begin{align*}
		M \otimes_A k(d) & \xrightarrow{\mathbbm{1} \otimes \varepsilon^\vee} M \otimes_A A^\vee(d) \\
		& \simeq M \otimes_A A \simeq M .
	\end{align*}
	Since the $ A $-module structure on the left hand side factors canonically through the counit $ A \to k $, the above is adjoint to a map $ \Nm_M: M \otimes_A k(d) \to \hom_A(k, M) $. 
	The \emph{Tate construction of $ M $} is the cofiber of the norm map 
	\begin{equation*}
		M^{tA} := \cofib(\Nm_M) .
	\end{equation*} 
	We use \emph{Tate cohomology of $ M $} to refer to the graded homotopy groups of the Tate construction on $ M $, and simply write ``Tate cohomology'' when $ M =k $ is assumed.  
\end{cons}

\subsection{Stable module categories.}\label{subsection:stablemodcats}
Given a finite-dimensional $ k $-algebra $ A $ and a perfect $ A $-module $ M $, its underlying $ k $-module is perfect. 
However, an $ A $-module whose underlying $ k $-module is perfect may not necessarily be perfect over $ A $. 
An example of this is given by $ A = k[x]/ (x^2) $ and the module $ k = A/(x) $; any projective resolution of $ k $ by $ A $ is necessarily unbounded. 
A measure of the discrepancy between these categories is given by the \emph{stable module category} of $ A $, which we construct in this section. 

We will see (Corollary \ref{cor:stmod_graded_modules}) that we could have defined the stable module category of $ A $ at the end of \S\ref{subsection:frobgor_duality}. 
However, the presentation of the stable module category in Definition \ref{defn:stmod} is crucial to our identification of an explicit generator (Proposition \ref{prop:stmod_generators}) and convenient exact sequences in $ \StMod_A $ which will ultimately yield sharper bounds for categorical polynomial entropy. 
Thus while the material here is not new, we include it for completeness, and for the unfamiliar reader. 

The following construction of the stable module category is borrowed from \cite[\S 2.1]{Mat15}, with minor changes. 
A more classical perspective can be found in \cite{MR4390795}.
\begin{cons}
	Let $ A $ be a Frobenius ring (Definition \ref{defn:frobeniusalg}). 
	The category $ \Mod_A^\heartsuit $ of all discrete $ A $-modules admits a combinatorial model structure \cite[\S 2.2]{MR1650134} with
	\begin{enumerate}
		\item fibrations given by surjections,
		\item cofibrations given by injections, and
		\item weak equivalences given by stable equivalences. 
		Given a morphism $ f: V \to W $ in $ \Mod_A^\heartsuit $, $ f $ is a \emph{stable equivalence} if there exists $ g: V \to W $ such that 
		\begin{equation*}
		\id_V - f\circ g: V \to V \qquad \text{ and } \qquad \id_W - g \circ f: W \to W  
		\end{equation*}
		each factor through a projective $ A $-module.
	\end{enumerate}
\end{cons}

\begin{defn}\label{defn:stmod}
	The \emph{big stable module category} $ \StMod_A^\mathrm{big} $ is given by the $ \infty $-categorical localization $ \LMod_A^\heartsuit [\mathcal{W}^{-1}] $ where $ \mathcal{W} \subset \LMod_A $ are the weak equivalences \cite[Definition 1.3.4.15]{HA}. 
	The \emph{small stable module category} $ \StMod_A $ is defined to be compact objects in the big stable module category. 
\end{defn}
\begin{ntn}
	If $ M, N \in \StMod_A $, we write $ \StExt^{s,t}_A(M,N) = \pi_{-s,-t}\hom_{\StMod_A}(M,N) $. 
\end{ntn}
By \cite[Prop. 4.1.7.4]{HA}, $ \StMod_A $ inherits a symmetric monoidal structure from $ \Mod_A $. 
Hovey shows \cite[\S2.2]{MR1650134} that the homotopy category $ \ho \StMod_A^\mathrm{big} $ of $ \StMod_A^\mathrm{big} $ is the classical stable module category:
\begin{equation*}
	\hom_{\ho\StMod_A^\mathrm{big}}(V, W) = \hom_{\Mod^\heartsuit_A}(V, W)/\sim
\end{equation*}
where two morphisms $ \varphi,\psi $ are identified on the right hand side if their difference factors through a projective $ A $-module. 
\begin{prop}
	Let $ A $ be a Frobenius ring. 
	Then $ \StMod_A^{\mathrm{big}} $ is presentable. 
\end{prop}
\begin{proof}
	Follows from \cite[Proposition 1.3.4.22]{HA}.
\end{proof}
\begin{recollection} \cite[\S 3.2]{MNN} \label{recollection:A-1-localization}
	Let $ \cat $ be a presentable, symmetric monoidal, stable $ \infty $-category, and let $ A $ be an $ \E_1 $-algebra object in $ \cat $. 
	Let $ \cat_{A-\tors} $ denote the smallest stable subcategory of $ \cat $ containing $ A \otimes X $ and closed under colimits, for $ X $ dualizable in $ \cat $\footnote{This is also referred to as the localizing subcategory generated by $ A \otimes X $ for $ X $ dualizable.} \cite[Definition 3.1]{MNN}. 

	We say that an object $ X \in \cat $ is $ A^{-1} $-local if, for all $ Y \in \cat_{A-\tors} $, we have $ \hom_{\cat}(X,Y) \simeq 0 $ \cite[Definition 3.10]{MNN}. 
	We denote the full subcategory of $ A^{-1} $-local objects by $ \cat[A^{-1}] \subseteq \cat $. 
	The inclusion of $ A^{-1} $-local objects admits a right adjoint, which we refer to as $ L_{A^{-1}} $-localization. 
	An explicit formula computing the $ L_{A^{-1}} $-localization is described in \cite[Construction 3.12]{MNN}.
\end{recollection}
The following is a mild generalization of \cite[Theorem 2.4]{Mat15}. 
\begin{prop} \label{prop:stmod_LA1}
	Let $ k $ be a field, and $ A $ a finite-dimensional connected graded cocommutative Hopf algebra over $ k  $. 
	Then there is a symmetric monoidal equivalence of stable $ \infty $-categories 
	\begin{equation*}
		 \StMod_A^\mathrm{big} \simeq L_{A^{-1}}\Ind\left(\Mod_A\left(\Perf^{gr}_k\right)\right) .
	\end{equation*} 
\end{prop}
\begin{proof}
	We have an inclusion of finitely-generated discrete $ A $-modules into the $ \infty $-category of perfect graded $ k $-module spectra with an $ A $-module structure:
	\begin{equation*}
		\Mod^{\heartsuit, fg}_A  \subset \Mod_{A}\left(\Perf^{gr}_k\right) .
	\end{equation*}
	Taking Ind-completions \cite[\S5.3]{LurHTT} on both sides, we obtain
	\begin{equation*}
		\Phi: \Ind\left(\Mod^{\heartsuit, fg}_A\right) \simeq \Mod_{A}^\heartsuit \to \Ind \left( \Mod_{A}\left(\Perf^{gr}_k\right) \right)
	\end{equation*}
	which commutes with filtered colimits. 
	Composing $ \Phi $ with $ A^{-1} $-localization of Recollection \ref{recollection:A-1-localization} gives 
	\begin{equation*}
		\Mod_{A^\vee}^{\heartsuit,fg} \xrightarrow{\Phi} \Ind \left( \Mod_{A^\vee} \right) \xrightarrow{L_{A^{-1}}} L_{A^{-1}}\Ind\left(\Mod_A\left(\Perf^{gr}_k\right)\right) .
	\end{equation*}
	Since this composite takes projective $ A $-modules to zero, it respects weak equivalences and descends to a symmetric monoidal functor
	\begin{equation*}
		\varphi: \Mod_{A}^\heartsuit[\mathcal{W}^{-1}] \simeq \StMod_A^\mathrm{big} \to L_{A^{-1}}\Ind\left(\Mod_A\left(\Perf^{gr}_k\right)\right) . 
	\end{equation*}

	We claim that $ \varphi $ commutes with homotopy colimits. 
	This boils down to the fact that $ \Phi $ takes short exact sequences of finitely-generated discrete $ A $-modules to cofiber sequences in $ \Mod_A $, and is defined to be left Kan extended from projective $ A $-modules. 
	Since $ \Mod_A(\Perf_k) $ is generated under colimits by $ \Mod_A(\Perf_k^\heartsuit)$, and $ \varphi $ preserves small colimits, we have shown that $ \varphi $ is essentially surjective.

	It remains to check that $ \varphi $ is fully faithful, which we check at the level of homotopy categories. 

	Since $ \hom(-, M) $ commutes with colimits by definition, it suffices to check that $ \varphi $ induces an equivalence on $ \pi_0 \hom(-, -) $ when the source is a finitely-generated $ A $-module.
	Note that $ \StMod_A $ is stable and the objects represented by finitely-generated $ A $-modules are compact, so we can also reduce to the case that the target is finitely-generated. 
	Finally, we use duality and exactness to reduce to showing that 
	\begin{equation*}
		\hom_{\ho \StMod_A}(k, M) \to \pi_0 \hom_{L_{A^{-1}}\Ind(\Mod_A(\Perf^{gr}_k))}(k, \varphi(M))
	\end{equation*}
	is an isomorphism for $ M =k(0)=: k $ the trivial $ A $-module. 

	The left-hand side of the above gives the Tate cohomology of $ A $ by \cite[Lemma 6.1.2]{MR4390795}.

	By Observation \ref{obs:modules_unit_invertible} and Proposition \ref{prop:finconnhopfimpliesfrobenius}, the $ A $-torsion objects and $ A^\vee \simeq A \otimes k(d) $-torsion objects agree as subcategories in $ \Ind(\Mod_A(\Perf^{gr}_k)) $, hence so do the associated localization functors $ L_{A^{-1}} \simeq L_{(A^\vee)^{-1}} $.
	By \cite[\S3.2]{MNN}, the $ (A^\vee)^{-1} $-localization of $ k $ is computed by the cofiber of a map
	\begin{equation}\label{eq:hopfalg_transfer}
		\colim_{\Delta^\op } \left| A^{\otimes \bullet +1} \right| \to k . 
	\end{equation} 
	We use that $ \hom_{\Ind(\Mod_A(\Perf^{gr}_k))}(\mathbbm{1}, -) =: \hom_A(\mathbbm{1},-) $ commutes with arbitrary colimits to rewrite 
	\begin{align*}
		\colim_{\Delta^\op} \pi_0\hom_A\left(\mathbbm{1}, A^{\otimes \bullet +1}  \right) 
		&\simeq \colim_{\Delta^\op_{\leq 1}} \pi_0\hom_A\left(\mathbbm{1}, A^{\otimes \bullet +1}  \right) \simeq \pi_0(k \otimes_A k) (d)
	\end{align*}
	where $ d $ is the Gorenstein parameter of $ A $ and we abuse notation by writing $ \pi_0(k \otimes_A k) $ to refer to its underlying graded $ k $-module. 
	It suffices to notice that the map (\ref{eq:hopfalg_transfer}) induces the norm map (\ref{cons:tate}) upon taking $ \Hom_A(\mathbbm{1}, -) $. 
\end{proof}

\begin{cor} \label{cor:stmod_graded_modules}
	Let $ A $ be a finite-dimensional connected graded cocommutative Hopf algebra over $ k $. 
	There is an equivalence of categories $ \StMod_A^{\mathrm{big}} \simeq \Mod_{k^{tA}}^\gr $. 
\end{cor}
\begin{proof}
	Follows from applying Proposition \ref{prop:gradedschwedeshipley} and observing that $ k(0) $ is a graded-generator of $ \StMod_A^{\mathrm{big}} $. 
	Finally, the proof of Proposition \ref{prop:stmod_LA1} shows that graded endomorphisms of $ k(0) $ in $ L_{A^{-1}}\Ind\left(\Mod_A\left(\Perf^{gr}_k\right)\right) $ agree with the Tate construction $ k^{tA} $ of Construction \ref{cons:tate}. 
\end{proof}
\begin{cons} [Twist dynamical system on $ \StMod_A $]	\label{cons:twist_dynamicalsystem}
	Let $ A $ be a finite-dimensional, graded connected Hopf algebra over $ k $. 
	The twist functor $ (1) $ of Recollection \ref{rec:graded_vs}\ref{recitem:twist_functor_vs} lifts to an autoequivalence $ \twist_A: \Ind \left( \Mod_{A}\left(\Perf^{gr}_k\right)\right) \to \Ind \left( \Mod_{A}\left(\Perf^{gr}_k\right)\right) $. 
	Because $ \twist $ preserves the subcategory $ \Mod_{A}\left(\Perf^{gr}_k\right) $ and the property of being a stable equivalence, the twist by 1 functor $ \twist_A $ descends to a well-defined endofunctor $ \twist_A: \StMod_A \to \StMod_A $. 
\end{cons}
The construction of $ \StMod_A $ allows for explicit identification of the loop functor \cite[\S2.2]{MR4390795}. 
\begin{obs} [Loops in $ \StMod_A $] \label{obs:stmod_loops}
	Let $ A $ be a $ k $-algebra and let $ M $ be a finitely-generated discrete $ A $-module. 
	Suppose given a surjection $ f \colon P \surj M $ where $ P $ is a projective graded $ A $-module. 
	Such an $ f $ always exists by Propositions \ref{prop:frob_algebras_cats} and \ref{prop:finconnhopfimpliesfrobenius}. 
	The short exact sequence of $ A $-modules $ \ker(f) \to P \to M $ induces an exact sequence in $ \Mod_A\left(\Perf_k\right) $ and in $ \StMod_A $. 
	However, by definition of the stable module category, $ P \simeq 0 $ in $ \StMod_A $. 
	Thus $ \ker(f) \simeq \loops M $. 
\end{obs}

\paragraph{A compact generator.} 
Here we identify a convenient `test object' in $ \StMod_A $ (see introduction to \S\ref{subsection:entropy_defn}). 

\begin{prop}\label{prop:stmod_generators}
	If $ A $ is a nontrivial finite-dimensional connected graded Frobenius algebra over a field $ k $, then $ \StMod_A^\omega $ has a generator (Recollection \ref{recollection:thick_subcategory}) given by a finite sum of twists of the ground field. 
\end{prop}
\begin{proof}
	We begin by observing that any discrete finite $ A $-module $ M $ can be written as an iterated extension of $ k(n) $ for $ n \in \Z $. 
	Let $ w = \min \{n \mid M_n \neq 0 \}$. 
	Since $ A $ is connected, the quotient of graded $ k $-vector spaces $ M \to k(w)^{\oplus \dim_k M_w} $ is canonically a morphism of $ A $-modules. 
	Then the exact sequence  
	\begin{equation}\label{eq:homogenization}
		\fib(p) \to M \xrightarrow{p} k(w)^{\oplus \dim_k M_w} 
	\end{equation}
	induces a cofiber sequence in $ \Mod_A\left( \Perf_k \right) $. 
	Replacing $ M $ by $ \fib(p) $ and iterating this procedure shows that any discrete finite $ A $-module is in the thick subcategory $ \langle k(n) \rangle_{n \in \Z} $. .
	The observation follows from observing that $ \Mod_A\left( \Perf_k \right)^\omega $ is generated under shifts and extensions by discrete $ A $-modules. 
	Since localization preserves compact generators, we conclude that 
	\begin{equation*}
		 \StMod_A^\omega = \left\langle k(\ell) \right\rangle_{\ell \in \Z} .
	\end{equation*} 
	Now take $ M = A, \fib(p), \ldots $ in the sequence (\ref{eq:homogenization}) and apply the previous procedure to produce a diagram
	\begin{equation*}
	\begin{tikzcd}
		k(d) = \fib(p_d) \ar[r] & \fib(p_{d-1}) \ar[d,"p_d"] \ar[r] &\cdots \ar[r] &\fib(p_0) \ar[r] \ar[d,"p_1"] & A \ar[d,"p_0"] \\
		& \bigoplus k(d-1) && \bigoplus k(1) & k(0)
	\end{tikzcd}
	\end{equation*} 
	where each pair of composable morphisms consisting of a horizontal arrow followed by a vertical arrow is a short exact sequence in $ \Mod_A^\heartsuit $. 
	These induce exact sequences in $ \Mod_A\left( \Perf_k \right)^\omega $.  
	Then by Recollection \ref{recollection:thick_subcategory},
	\begin{equation*}
		k(d) \in \langle \fib(p_{d-1}), k(d-1) \rangle \subseteq \langle \fib(p_{d-2}),k(d-2), k(d-1) \rangle \cdots \subseteq \langle k(0), \ldots, k(d-1) \rangle .
	\end{equation*} 
	Modifying the above argument slightly for $ A(-1) $, we see that $ k(-1) \in \langle k(0), \ldots, k(w-2) \rangle $, and similarly for negative twists of $ k $. 
\end{proof}
Next we show that an assumption for linear-algebraic bound on categorical polynomial entropy (Proposition \ref{prop:poincarepolybound}) holds for our categories of interest. 
\begin{prop}\label{prop:stmodproper}
	Let $ A $ be a finite-dimensional connected graded cocommutative Hopf algebra over a field $ k $. 
	Then $ \StMod_A $ is proper, i.e. all $ \StExt_A (-, -) $ are perfect complexes over $ k $. 
\end{prop} 
\begin{proof}
	Because twists of $ k(0) $ (and shifts thereof) generate $ \StMod_A $, it suffices to show that for each $ 0 \leq i, j <w $, $ \StExt_A(k(i), k(j)) $ is a bounded complex of finite dimensional $ k $-vector spaces. 
	Equivalently, for a fixed twist $ t = j-i $ the graded $ k $-vector space $ \bigoplus_s \StExt^{s, t}_A(k, k) $ is finite-dimensional.  
	By Construction \ref{cons:tate}, there is an exact sequence 
	\begin{equation*}
		k \otimes_A k(d) \to \hom_A(k, k) \to k^{tA}
	\end{equation*}
	of graded $ k $-module spectra.
	The result follows from considering the resulting long exact sequence of graded homotopy groups and observing that
	\begin{itemize} 
		\item For each $ t \in \Z $, $ \bigoplus_s \Ext^{s, t}_A(k, k) $ is finite-dimensional: 
		This follows from Theorem \ref{thm:ext_fingen} and Observation \ref{obs:homogenous_generators_bidegree}.  
		\item For each $ t \in \Z $, $ \bigoplus_s \pi_s (k \otimes_A k)_t $ is finite-dimensional:
		Since $ k $ is a field, duality interacts nicely with taking homotopy groups and we can understand the first term as follows: 
		\begin{align*}
			\pi_{s,t} (k \otimes_A k) &\simeq \pi_{-s,-t}\hom_k(k \otimes_A k, k) \\
			&=\simeq \pi_{-s,-t}\hom_A(k, k) = H^{s}(A; k)_t
		\end{align*}
		hence the result follows from the previous case. \qedhere 
	\end{itemize}
\end{proof}
\begin{rmk}
	The preceding results are where working with graded modules and the assumption that $ A $ is connected become essential. 
\end{rmk}

\section{Entropy of the twist functor}\label{section:twist_entropy_results}
In this section, we arrive at a proof of the main theorem. 
We state general results for controlling complexity in graded stable $ \infty $-categories in \S\ref{subsection:general_bounds}, then apply them to the twist dynamical system on stable module categories in subsequent sections. 
We show that the categorical entropy of the twist functor $ \twist $ on $ \StMod_A $ vanishes. 
We show that the categorical polynomial entropy of the twist functor $ \twist $ is bounded below by one less than the Krull dimension of the cohomology of $ A $ using a linear-algebraic bound of Fan-Fu-Ouchi (Proposition \ref{prop:poincarepolybound}). 
Since we are not able to show that $ \StMod_A $ is smooth (Definition \ref{def:smoothpropcat}) for generic $ A $, providing sharp upper bounds for the categorical polynomial entropy of $ \twist $ is a more delicate matter. 
We show that a careful application of the procedure employed in the proof of Proposition \ref{thm:twist_expentropy_iszero} can lead to upper bounds for the categorical polynomial entropy of $ \twist $. 

\subsection{Generalities.}\label{subsection:general_bounds} 
We start by introducing a general-purpose strategy for computing the complexity of the twist functor on locally graded stable $ \infty $-categories (Definition \ref{defn:loc_gr_pres_cats}). 
Informally a locally graded stable $ \infty $-category $ \cat $ is a stable $ \infty $-category $ \cat $ equipped with an automorphism $ (1): \cat \to \cat $. 
The reader is invited to take $ \cat = \StMod_A $ or $ \cat = \Mod_A\left( \Perf_k\right) $ and $ (1) = \twist $ in the following to fix ideas. 
\begin{defn}\label{defn:ladder_pyramid}
	Let $ \cat $ be a graded stable $ \infty $-category, and let $ M, N \in \cat $. 
	Say that $ M $ \emph{is a staircase for} $ N $ if there is a resolution (Definition \ref{defn:resolution})
	\begin{equation}\label{eq:ladder_sequence}
		N(r)[s] \longrightarrow M(r-t_m)[s_m] \longrightarrow \cdots \longrightarrow M(r-t_1)[s_1] \longrightarrow N
	\end{equation}
	for some integers $ m, r, t_i \in \Z_{>0} $ and $ s_i, s \in \Z $ such that $ t_i \leq r $. 
	We will refer to $ m $ as the \emph{height} of the staircase. 

	We say that an $ \ell $-tuple $ (N_1, \ldots, N_\ell) $ of objects in $ \cat $ is a \emph{pyramid} if each $ N_i $ is a staircase for $ N_{i+1} $. 
	We will refer to $ \ell $ as the \emph{dimension} of the pyramid. 

	Let $ \cat $ be a locally graded stable $ \infty $-category, and let $ X \in \cat $ be an object. 
	The object $ X $ is \emph{periodic} if there exist integers $ a, b $, $ a \neq 0 $ such that $ X(a)[b] \simeq X $. 
\end{defn}
\begin{obs}
	Let $ \cat = \Mod_A\left(\Perf_k^\gr\right) $ for some connected graded cocommutative Hopf algebra $ A $ over $ k $. 
	If $ A $ is a staircase for $ M $ in $ \Mod_A\left(\Perf_k^\gr\right) $, then $ M $ is periodic in $ \StMod_A $. 

	Given a staircase for $ N $ of height $ r $, we can build a staircase of height $ kr $ for arbitrary $ k \in \Z_{> 0} $ by splicing twists of the given staircase. 
	We illustrate the case $ k= 2 $. 
	\begin{equation*}
	\begin{tikzcd}[column sep=small]
		N(2r)[2s] \ar[r] & M(2r-t_m)[s+s_m] \ar[r]& \cdots \ar[r] & M(2r-t_1)[s+s_1] \ar[lld] &\\
		 & M(r-t_m)[s_m] \ar[r] & \cdots \ar[r] & M(r-t_1)[s_1] \ar[r]& N 
	\end{tikzcd}
	\end{equation*}
	where the top row is $ -(r)[s] $ applied to the first $ m+1 $ terms of (\ref{eq:ladder_sequence}), the bottom row is the last $ m+1 $ terms of (\ref{eq:ladder_sequence}), and the diagonal arrow is defined to be the composite 
	\begin{equation*} 
		{M(2r-t_1)[s+s_1] \to N(r)[s] \to M(r-t_m)[s_m]} .
	\end{equation*} 
\end{obs}
\begin{prop}\label{prop:complexitybootstrap}
	Let $ \cat $ be a locally graded stable $ \infty $-category, and let $ G \in \cat $ be a compact generator. 
	Suppose given a pyramid $ (N_1, \ldots, N_{\ell}) $ in $ \cat $ such that $ N_1 $ is periodic, i.e. $ N_1(a)[b] \simeq N_1 $ where $ a, b $ are integers, $ a \neq 0 $. 
	Then there exists a constant $ C \gg 0 $ such that for $ n \gg 0 $, 
	\begin{equation*} 
		\delta_0(G, N_\ell(n)) < C \cdot n^{\ell-1} \qquad n \gg 0.
	\end{equation*} 
\end{prop}
\begin{proof}
	We prove this by induction on the dimension $ \ell $ of our pyramid. 

	Base case: Let $ \ell = 1 $. 
	By assumption, there is an equivalence $ N_1(a)[b] \simeq N_1 $. 
	Combining this with Observation \ref{obs:deltazeroshiftinsensitive}, we have 
	\begin{align*} 
		\delta_0(G, N_1(n)) = \delta_0\left(G, N_1\left(n - \floor*{\frac{n}{a}}a\right)\right) \leq \max_{0 \leq t < a} \delta_0(G, N_1(t)) 
	\end{align*} 
	which proves the claim. 

	Inductive step: We assume the statement for $ \ell-1 \geq 1 $.  
	We apply Lemma \ref{lemma:resolution_subadditive} to our assumption of the existence of a resolution (\ref{eq:ladder_sequence}) to  
	\begin{align}
		\delta_0(G, N_\ell(n)) &\leq \sum_{j=1}^{m} \delta_0(G, N_{\ell-1}(n- t_j)) + \delta_0(G, N_\ell(n-r)) \\
		& \leq \sum_{j=1}^{m} \sum_{k_j=1}^{\floor*{\sfrac{n}{t_j}}} \delta_0(G, N_{\ell-1}(n-k_jt_j)) + \max_{0 \leq i < r} \delta_0(G, N_\ell(i))  \label{eq:complexitybootstrap_indstep}
	\end{align}
	The inductive assumption implies that $ \delta_0(G, N_{\ell-1}(n)) \leq C \cdot n^{\ell-2} $ for $ n \gg 0$. 
	It follows from Lemma \ref{lemma:polynomialsums} that there exists a constant $ C' \gg 0 $ such that the first sum of (\ref{eq:complexitybootstrap_indstep}) is $ < C'\cdot n^{\ell-1} $ for $ n \gg 0 $ as desired.
\end{proof}

\begin{lemma}\label{lemma:polynomialsums} 
	Let $ P(x) $ be a degree $ \ell $ polynomial with nonnegative real-valued coefficients, and fix a positive integer $ d \in \Z_{> 0} $. 
	Then for any $ n \in \Z_{>0 }$, we have an inequality  
	\begin{equation}\label{equation:polynomial_descending_sum}
		P(n) + P(n-d) + P(n-2d) + \cdots + P\left( n-d\floor*{\frac{n}{d}} \right) < C \cdot n^{\ell+1} \qquad \text{for} \qquad n \gg 0 .
	\end{equation}
\end{lemma}
\begin{proof}
	Observe that the statement for $ P(x) = x^e $, $ 0 \leq e \leq \ell $ implies the statement for general $ P $, so we may assume without loss of generality that $ P(x) = x^\ell $ is a monomial. 
	Write $ S_P(n, d) $ for the sum in (\ref{equation:polynomial_descending_sum}).  
	By our assumption on $ P $ we have $ S_P(n, d) \leq S_P(n, 1) $, so it suffices to prove the case $ d = 1 $. 
	This is a consequence of Faulhaber's formula \cite{faulhaber1631academia}:
	\begin{equation*}
		\sum_{k=1}^n k^d = \frac{n^{d+1}}{d+1} + O(n^d) . \qedhere
	\end{equation*} 
\end{proof}
We can give an explicit condition for when a staircase in a pyramid does not contribute to the categorical polynomial entropy. 

\begin{obs}\label{obs:thicksubcat_complexity}
	Let $ \cat $ be a stable $ \infty $-category, and let $ M, N \in \cat $. 
	Suppose $ N $ is in the thick subcategory generated by $ M $. 
	Then $ \delta_0(M, N) <\infty $.  
\end{obs}
Observation \ref{obs:thicksubcat_complexity} leads to the following refinement of Proposition \ref{prop:complexitybootstrap}.
\begin{prop}\label{prop:complexitybootstrap_refined}
	Let $ \cat $ be a locally graded stable $ \infty $-category, and let $ G \in \cat $ be a generator. 
	Suppose given a pyramid $ (N_1, \ldots, N_{\ell}) $ in $ \cat $ such that $ N_1 $ is periodic, i.e. $ N_1(a)[b] \simeq N_1 $ where $ a, b $ are integers, $ a \neq 0 $. 
	Assume that there exist integers $ b_i $, $ 1 \leq b_1 < \cdots b_h < \ell $, such that $ N_{b_i +1} $ is in the thick subcategory generated by $ \bigoplus_{t \in T} N_{b_i}(t) $.
	Then there exists a constant $ C \gg 0 $ such that for $ n \gg 0 $, 
	\begin{equation*} 
		\delta_0(G, N_\ell(n)) < C \cdot n^{\ell-1-h} \qquad n \gg 0.
	\end{equation*} 
\end{prop}
\begin{proof}
	The result follows from a modification of the inductive step of Proposition \ref{prop:complexitybootstrap} when $ \ell = b_i + 1 $. 
	We demonstrate the case $ \ell = b_1 +1 $; the cases $ \ell =  b_i + 1 $ are proved similarly. 
	By our assumption, Observation \ref{obs:thicksubcat_complexity}, and (\ref{lemmaitem:complexity_triangleineq}) and (\ref{lemmaitem:complexity_subadditive}) of Lemma \ref{lemma:complexityprops},    
	\begin{align}\label{eq:complexitybootstrap_refined_indstep}
		\delta(G, N_\ell(n)) &\leq \delta(G, \bigoplus_{t \in T}N_{b_1}(n+t)) \cdot \delta(\bigoplus_{t \in T}N_{b_1}(n+t), N_{b_1 +1}) \leq C' \sum_{t\in T} \delta(G, N_{b_1}(n+t)) 
	\end{align}
	The inductive assumption implies that $ \delta(G, N_{b_1}(n)) \leq C \cdot n^{b_1-1} $ for $ n \gg 0$. 
	Since the sum on the right hand side of (\ref{eq:complexitybootstrap_refined_indstep}) is finite and independent of $ n $, there exists a constant $ C'' \gg 0 $ such that $ \delta(G, N_{b_1 +1}(n)) < C''\cdot n^{b_1-1} $ for $ n \gg 0 $ as desired.
\end{proof}
We identify special case in which the assumption of Proposition \ref{prop:complexitybootstrap_refined} holds.
\begin{prop}\label{prop:nilextensions}
	Let $ \cat $ a locally graded stable $ \infty $-category. 
	Let $ M $ be a staircase for $ N $ (Definition \ref{eq:ladder_sequence}). 
	Suppose that the map classifying the resolution (\ref{eq:ladder_sequence}) is nilpotent. 
	Then $ M $ is in the thick subcategory generated by $ \bigoplus_{j \in S} N(j) $ for some finite subset $ S \subset \Z $.  
\end{prop}
\begin{rmk}
	This proposition is similar in spirit to the ideas contained in \cite[\S4]{MNN} (in particular Proposition 4.7 of \emph{loc. cit.}), though our result does not require $ M $ to be an algebra object (see Observation \ref{obs:choiceoftower}).
\end{rmk}
\begin{proof}
	[Proof of Proposition \ref{prop:nilextensions}]
	The resolution   
	\begin{equation}
	\begin{tikzcd}
		N(r)[s] \ar[r,"g_m"] & M(r-t_m)[s_m] \ar[r,"{g_{m-1}}"] & \cdots \ar[r,"g_1"] & M(r-t_1)[s_1] \ar[r,"g_0"] & N
	\end{tikzcd}
	\end{equation}
	is classified by a map $ \varphi \colon N \to N(r)[s+m] $ which is defined to be the composite
	\begin{equation*}
	\begin{tikzcd}
		N \ar[r,"f_0"] & \Sigma \fib(g_{0}) \simeq \Sigma \cofib(g_{1}) \ar[r,"f_{1}"] & \Sigma Z_{1} \ar[r] & \cdots \ar[r,"f_{m}"] & \Sigma^m N(r)[s] .
	\end{tikzcd}
	\end{equation*}

	Suppose $ \varphi^e $ is nullhomotopic for some $ e \in \Z_{>0} $. 
	For $ 0 < b \leq e $ and $ 0 < c < m $, we contemplate the following commutative diagram in $ \cat $ in which all rows and columns are cofiber sequences
	\begin{equation}\label{eq:nilpotent_spiral_sequence}
	\begin{tikzcd}[ampersand replacement=\&]
		N \ar[d,"f_{c}\cdots f_0\circ\varphi^{b}"'] \ar[r, equals] \& N \ar[d, "f_{c+1}\cdots f_0\circ\varphi^b"] \ar[r] \& 0 \ar[d] \\
		Z_c (br-T_{c+1})[b(s+m)+c] \ar[r,"f_{c+1}"] \ar[d] \& Z_{c+1} (br-T_{c})[b(s+m)+c+1] \ar[r] \ar[d] \& \cofib(f_{c+1}) (br-T_c)[b(s+m)+c+1] \ar[d, equals] \\
		\cofib(f_{c}\cdots f_0\circ\varphi^{b}) \ar[r] \& \cofib(f_{c+1}\cdots f_0\circ\varphi^b) \ar[r] \& \cofib(f_{c+1}) (br-T_c)[b(s+m)]
	\end{tikzcd} 
	\end{equation}
	where we write $ T_c = \sum_{i=c}^m t_i $. 

	We observe that
	\begin{itemize}
		\item For each $ 0<c< m $, $ \cofib(f_c) $ is a twist of $ M $ by definition of $ f $.
		\item The bottom row in the diagram (\ref{eq:nilpotent_spiral_sequence}) gives 
		\begin{equation*}
			\cofib(f_{c+1}\cdots f_m\circ\varphi^{b}) \in \left\langle \cofib(f_{c}\cdots f_m\circ\varphi^b), \cofib(f_c) (br)\right\rangle .
		\end{equation*}
		\item For each $ 0< b \leq e $, we have $ \cofib(\varphi^b) \in \left\langle \cofib(f_m)(B), \cofib(f_{m-1} \cdots f_0\varphi^{b-1})  \right\rangle $ for some $ B \gg 0$.
		\item Taking $ b = e $, we have $  N[1] \oplus N[e(s+m)] \simeq \cofib(\varphi^e) $.
	\end{itemize}	
	The result follows from a `spiral' descending induction.	
\end{proof}
\begin{rmk}
	The crucial difference between the proof of Proposition \ref{prop:complexitybootstrap_refined} and the proof of Proposition \ref{prop:complexitybootstrap} is: The length of the sum in the inequality (\ref{eq:complexitybootstrap_indstep}) is approximately linear in $ n $, but in (\ref{eq:complexitybootstrap_refined_indstep}), the length of the sum does not depend on $ n $. 
	Consequently, the asymptotic growth rate of $ \delta(G, M(n)) $ is controlled by the asymptotic growth rate of $ \delta(G,N(n)) $ times a \emph{constant} factor instead of a factor linear in $ n $. 
\end{rmk}

\subsection{Categorical entropy of \textsf{tw}.}\label{subsection:catenttwistbounds}
In this section, we apply the general purpose tools of \S\ref{subsection:general_bounds} to show that $ h_{\mathrm{cat}}(\twist) $ vanishes. 
First, we introduce a notion which will allow us to construct pyramids in $ \StMod_A $.
\begin{defn}\label{defn:hopfalg_tower}
Let $ A $ be a graded Hopf algebra over a field $ k $. 
A \emph{tower under $ A $} is a commutative diagram (\ref{eq:hopfalg_tower}) of Hopf algebras 
	\begin{equation}\label{eq:hopfalg_tower}
	\begin{tikzcd}
		M_0 \ar[r] & A = A_0 \ar[d] \\
		M_{1} \ar[r] & A_{1} \ar[d, phantom, "\vdots"] \\
		M_{\ell} \ar[r] & A_\ell \ar[d] \\
		& A_{\ell+1} 
	\end{tikzcd}
	\end{equation}
	satisfying the conditions 
	\begin{enumerate}
		\item All composites consisting of a right arrow followed by a down arrow are extensions of Hopf algebras (Definition \ref{defn:hopf_alg_extension}).
		We refer to each such pair of composites as a \emph{storey}, and say that the pyramid (\ref{eq:hopfalg_tower}) has $ \ell + 1 $ storeys. 
		\item All $ M_i $ are monogenic. 
		\item The final term $ A_{\ell+1} = k $ is the trivial Hopf algebra.  
	\end{enumerate} 
\end{defn}
\begin{theorem}\label{thm:twist_expentropy_iszero}
	Let $ A $ be a nontrivial connected, graded, cocommutative Hopf algebra which is finite-dimensional over a field $ k $. 
	Let $ \twist: \StMod_A \to \StMod_A $ denote the twist functor of Construction \ref{cons:twist_dynamicalsystem}. 
	Then $ h_{\mathrm{cat}}(\twist) = 0 $.
\end{theorem}
\begin{proof} 
	By Proposition \ref{prop:stmod_generators}, we can take our generator of $ \StMod_A $ to be $ G = \oplus_{i=0}^{d-1} k(i) $ where $ d $ is the Gorenstein parameter of $ A $. 
	Notice that
	\begin{align*}
		h_{\mathrm{cat}}(\twist) &= \lim_{n \to \infty} \frac{\log \delta_0(G, G(n))}{n} \\
		&\leq \lim_{n \to \infty} \frac{\log \sum_{j=0}^{d-1} \delta_0(G, k(n+j))}{n} \qquad \text{by subadditivity (Lemma \ref{lemma:complexityprops})}  
	\end{align*}
	Thus by Proposition \ref{prop:complexitybootstrap}, it suffices to exhibit a pyramid $ (A_1,\ldots, A_{\ell +1}) $ in $ \StMod_A $ such that $ {A_{\ell +1} = k(0)} $. 
	This follows from Lemma \ref{lemma:towertopyramid}. 
\end{proof}
\begin{lemma}\label{lemma:towertopyramid}
	Let $ A $ be a connected, graded, cocommutative Hopf algebra which is finite-dimensional over a field $ k $. 
	Suppose given a tower under $ A $ (Definition \ref{defn:hopfalg_tower}) with $ \ell+1 $ storeys.  
	Then there is a pyramid in $ \StMod_A $ of dimension $ \ell +1 $. 
\end{lemma}
\begin{proof}
	By Propositions \ref{prop:structure_hopfalg_charzero} and \ref{prop:existence_monogen_subhopf} (corresponding to the cases $ \mathrm{char}\; k = 0 $ and $ \mathrm{char}\; k >0 $, resp.), there exists a tower (Definition \ref{defn:hopfalg_tower}) under $ A $. 
	Moreover, all $ M_i $ are monogenic elementary of minimal height, i.e. of the form in Examples \ref{ex:monogenichopfalgs}. 
	Then by Lemma \ref{lemma:periodic_monogenic_modules}, for each $ 0 \leq i \leq \ell $ we have a resolution  
	\begin{equation*}
	\begin{tikzcd}[ampersand replacement=\&, row sep=tiny]
		k(m_i) \ar[r] \& M_i(n_i) \ar[r] \& M_i(0) \ar[r] \& k(0)
	\end{tikzcd}
	\end{equation*}
	in $ \Mod_{M_i}\left(\Perf_k\right) $ for $ n_i > 0 $. 
	Pushing forward to $ \Mod_{A_i}\left(\Perf_k\right) $, this induces a resolution in $ \Mod_{A_i}\left(\Perf_k\right) $
	\begin{equation} \label{eq:inductive_periodicity_sequences}
	\begin{tikzcd}[row sep=tiny, ampersand replacement=\&]
		k(m_i) \otimes_{M_i} A_i \ar[r] \ar[d, equals] \&  M_i(n_i) \otimes_{M_i} A_i \ar[r] \ar[d, equals] \& M_i \otimes_{M_i}A_i \ar[d, equals]\ar[r] \& k \otimes_{M_i} A_i \ar[d, equals] \\
		A_{i+1}(m_i) \ar[r] \& A_i(n_i) \ar[r] \& A_i \ar[r] \& A_{i+1} 
	\end{tikzcd}
	\end{equation}
	where we have used the identification of Remark \ref{rmk:hopf_alg_quotient_in_extension}. 
	Since the forgetful functors are exact, we regard (\ref{eq:inductive_periodicity_sequences}) as a resolution in $ \Mod_{A}\left(\Perf_k\right) $. 
	Thus $ (A_1, \ldots, A_{\ell+1}) $ forms a pyramid in $ \StMod_A $ such that $ A_{\ell+1} = k(0) $ and $ A_1 $ is periodic. 
\end{proof}

\begin{lemma}\label{lemma:periodic_monogenic_modules}
	Let $ M $ be a finite graded monogenic Hopf algebra over a field $ k $ of arbitrary characteristic. 
	Then there is a resolution of the following form in $ \Mod_M \left( \Perf_k \right) $
	\begin{align}
	\begin{tikzcd}[ampersand replacement=\&]
		k(m) \ar[r] \& M(n) \ar[r] \& M(0) \ar[r] \& k(0)
	\end{tikzcd}\label{eq:proto_inductive_periodicity_seq2}
	\end{align}
	for some positive integers $ m>n> 0 $.
\end{lemma}
\begin{proof}
	When a morphism $ \psi: P \to Q $ of $ M $-modules in $ \Mod_M^{\heartsuit} $ is surjective (resp. injective), then we have equivalences
	\begin{align*} 
		\fib( \psi) &\simeq \ker(\psi) \\
		\left(\cofib (\psi) \right.&\simeq \left.\coker(\psi) \right),
	\end{align*}
	where the fiber (resp. cofiber) are computed in the derived category $ \Mod_M $ and the kernel (resp. cokernel) are computed in the abelian 1-category $ \Mod_M^\heartsuit $. 
	Write $ x \in M $ for a generator of $ M $ and let $ n = |x| $ and $ h =$ the height of $ x $. 
	The result follows from noting that the kernel (resp. cokernel) of multiplication $ \cdot x: M(n) \to M $ by $ x $ is $ k(nh) $ (resp. $ k(0) $). 
\end{proof}

\subsection{Polynomial entropy of \textsf{tw}.} \label{subsection:polyenttwistbounds}
We will need slightly different techniques to establish upper and lower bounds on $ h_{pol}(\twist) $. 
\begin{prop}\label{prop:upperboundres_length}
	Let $ A $ be a connected, graded, cocommutative Hopf algebra which is finite-dimensional over a field $ k $. 
	Let a pyramid $ (N_1, \ldots, N_{\ell}) $ in $ \StMod_A $ such that $ N_1 $ is periodic and $ N_{\ell} = k(0) $ be given. 
	Then 
	\begin{equation*} 
		h_{pol}(\twist) \leq \ell-1 . 
	\end{equation*}  
\end{prop}
\begin{proof}
	We compute
	\begin{align*}
		h_{\pol}(\twist) &= \limsup_{n \to \infty} \frac{\log \delta_0(G, G(n))}{\log n} \\
		&\leq \limsup_{n \to \infty} \frac{\log \sum_{j=0}^d \delta_0(G, k(n+j))}{\log n} \qquad &\text{by subadditivity (Lemma \ref{lemma:complexityprops})}	\\
		&\leq \limsup_{n \to \infty} \frac{\log C n^{\ell-1}}{\log n} \qquad \qquad &\text{ by Proposition \ref{prop:complexitybootstrap}} \\
		&= \ell-1 . & & \qedhere
	\end{align*}
\end{proof}
\begin{cor}
	Let $ A $ be a connected, graded, cocommutative Hopf algebra which is finite-dimensional over a field $ k $. 
	Suppose given a tower under $ A $ (Definition \ref{defn:hopfalg_tower}) with $ \ell+1 $ storeys.  
	Then the categorical polynomial entropy of $ \twist: \StMod_A \to \StMod_A $ satisfies
	\begin{equation*}
		h_\pol(\twist) \leq \ell .
	\end{equation*}
\end{cor}
\begin{obs}\label{obs:choiceoftower}
	The \emph{existence} of a tower such as in diagram (\ref{eq:hopfalg_tower}) gives us \emph{a} bound for the categorical polynomial entropy and allows us to show that the categorical entropy vanishes identically. 
	However, notice that the proofs of Proposition \ref{prop:upperboundres_length} and Theorem \ref{thm:twist_expentropy_iszero} do not require each `stage' of the tower to be a central extension, or that they be extensions \emph{of Hopf algebras}. 
	In fact, it suffices to find $ A $-modules $ M_i $ satisfying relations such as those in diagram (\ref{eq:inductive_periodicity_sequences}). 
	Thus we should not expect a bound obtained in this way to be optimal.  
\end{obs} 
Our proof of the lower bound is similar in spirit to that of \cite[Proposition 6.4]{MR4233273}. 
\begin{prop}\label{prop:krulldimlowerbound}
	Let $ A $ be a nontrivial connected, graded, (co)commutative Hopf algebra which is finite-dimensional over a field $ k $. 
	Let $ \twist: \StMod_A \to \StMod_A $ denote the twist functor. 
	Then the categorical polynomial entropy of the twist functor admits the lower bound 
	\begin{equation*}
		h_{\pol}(\twist) \geq \dim H^*(A; k) - 1. 
	\end{equation*}
\end{prop}
\begin{proof}
By Proposition \ref{prop:poincarepolybound}, \ref{prop:stmodproper} and Theorem \ref{thm:twist_expentropy_iszero}, we have an inequality
\begin{equation}\label{eq:hpol_twist_ext}
	h_{\pol}(\twist) \geq \limsup_{n  \to \infty} \frac{1}{\log N} \left( \log \sum_{\ell} \dim_k \Ext^\ell_{\StMod_A}(G, G(n)) \right)
\end{equation}
for $ G $ some generator of $ \StMod_A $. 
Let $ G = \bigoplus_{i=0}^d k(i) $ be the generator of Proposition \ref{prop:stmod_generators}. 
Then each term in (\ref{eq:hpol_twist_ext}) splits as 
\begin{equation*}
	\StExt^\ell(G, G(n)) = \StExt^{\ell,-n}(G, G) = \bigoplus_{i, j = 0}^d \StExt_A^{\ell, i-j-n}(k, k) .
\end{equation*} 
The limit on the right hand side of (\ref{eq:hpol_twist_ext}) is bounded from below by  
\begin{equation*} 
	\limsup_{n \to \infty} \frac{\log C(n)}{\log n} \qquad \text{where} \qquad C(n) := \sum_{\ell} \dim_k \StExt^{\ell,-n}(k, k) .
\end{equation*} 

The proof of Proposition \ref{prop:stmodproper} in particular implies that $ \StExt^{s,t}_A = \Ext^{s, t}_A $ for $ s<0$ and $ t<0 $.  

By Theorem \ref{thm:ext_fingen}, $ H^*(A;k) $ is a finitely-generated graded $ k $-algebra. 
Write $ d $ for the Krull dimension of $ H^*(A;k) $. 
By graded Noether normalization \cite[Theorem 1.5.17]{BH97} and Observation \ref{obs:homogenous_generators_bidegree}, $ H^*(A;k) $ is a finite module over $ k[x_1, \ldots, x_d] $ where the $ x_i $ are homogeneous elements of bidegree $ |x_i| = (\ell_i, m_i) $ where $ \ell_i, m_i \in \Z_{>0} $.   
Write $ P_n(t)$ (resp. $ Q_n(t) $) for the Poincar\'e series of $ H^*(A; k)_n $ and (resp. $ k[x_1, \ldots, x_d]_n $). 
We observe that that $ C(n) = P_n(0) $.  
We compute
\begin{align*}
	h_\pol(\twist) \geq \limsup_n \frac{\log P_n(0)}{\log n}  = \limsup_n \frac{\left(\log f(0)+ \log Q_n(0) \right)}{\log n}  \qquad \text{by Lemma \ref{lemma:finite_module_poincare}.} 
\end{align*}
Then $ Q_n(0) $ is the number of partitions of $ n $ into parts belonging to $ \{\ell_1, \ldots, \ell_d \}$. 
By Lemma \ref{lemma:gradedpartitions}, we have $ \limsup_n \frac{1}{\log n} \left(\log Q_n(0) \right) \geq d-1 $. 
\end{proof}
\begin{lemma}\label{lemma:finite_module_poincare}
	Let $ R = k[x_1, \ldots, x_m] $ be a bigraded polynomial ring on homogeneous generators, and let $ M $ be a finite bigraded module over $ R $. 
	Write $ P(t) := \sum_{i, j \in \Z} \dim_k R_{i,j}t^i $ and $ P_j(t) := \sum_{i \in \Z} \dim_k R_{i,j}t^i $. 
	Write $ Q(t), Q_j(t) $ for the analogous power series corresponding to $ M $. 
	Then there exists a \emph{polynomial} $ f \in \Z[t] $ such that $ Q(t) = f \cdot P(t) $ and $ Q_j(t) = f \cdot P_j(t) $. 
\end{lemma}
\begin{proof}
	Induct on $ m $ as in the proof of \cite[Lemma 2.6]{MR298694}. 
\end{proof}
\begin{lemma}\label{lemma:gradedpartitions}
	\cite[Thm 15.2]{Nathanson}
	Let $ W $ be a nonempty set of positive integers such that $ gcd(W) = 1 $ and $ W $ has cardinality $ k $. 
	Then if $ p_W(n) $ is the number of partitions of $ n $ into parts belonging to $ W $, then
	\begin{equation*}
		p_W(n) = \left(\prod_{a \in W} a \right)^{-1} \frac{n^{k-1}}{(k-1)!} + O(n^{k-2}). 
	\end{equation*}
\end{lemma}

\begin{warning}\label{warning:genhypothesisfails}
	The reader might be tempted to relate $ \pi_{*, *}k^{tA} $-modules to coherent sheaves on weighted projective space of dimension one less than the Krull dimension of $ \pi_{*, *}k^{tA} $ by analogy with Proposition \ref{prop:stmod_Dbcoh}, appeal to \cite[Prop. 3.2]{CK08} which shows that the bounded derived category of said weighted projective space is smooth in the sense of Definition \ref{def:smoothpropcat}, and deduce equality in Proposition \ref{prop:krulldimlowerbound}. 
	However, in general the functor which takes homotopy groups 
	\begin{equation*}
		\pi_{*,*}: \Mod_{C^*(A; k)} \simeq \Mod_{A} \to \Mod_{H^*(A; k)}  
	\end{equation*}
	is \emph{not} an equivalence of categories: To assert that this functor is an equivalence categories is a version of Freyd's generating hypothesis, which is known not to hold without restrictive conditions on $ A $ (cf. \cites[Remark 2.22]{BHLSZ21}[Theorem 1.1]{BCC07}[Theorem A]{MR2277690}). 
\end{warning}
Collectively, the results of this section give both an upper bound and a lower bound for $ h_\pol (\twist) $. 
\begin{theorem}\label{thm:catpol_bounds}
	Let $ A $ be a nontrivial connected, graded, (co)commutative Hopf algebra which is finite-dimensional over a field $ k $. 
	Let $ \twist: \StMod_A \to \StMod_A $ denote the twist functor. 
	Then the categorical polynomial entropy of the twist functor admits the lower bound 
	\begin{equation*}
		h_{\pol}(\twist) \geq \dim H^*(A; k) - 1. 
	\end{equation*}
	Suppose given a tower under $ A $ (Definition \ref{defn:hopfalg_tower}) with $ \ell+1 $ storeys.  
	Then the categorical polynomial entropy of $ \twist: \StMod_A \to \StMod_A $ satisfies
	\begin{equation*}
		h_\pol(\twist) \leq \ell .
	\end{equation*}
\end{theorem}

\subsection{Refining the upper bound.}
We give an explicit conditions for when a storey of the tower (\ref{eq:hopfalg_tower}) does not contribute to the categorical polynomial entropy. 

\begin{prop}\label{prop:upperboundres_length_refined}
	Let $ A $ be a nontrivial connected, graded, cocommutative Hopf algebra which is finite-dimensional over a field $ k $. 
	Let a tower under $ A $ with $ \ell +1 $ storeys as in diagram (\ref{eq:hopfalg_tower}) be given.   
	Assume that there exist integers $ b_i $, $ 1 \leq b_1 < \cdots b_h < \ell $, such that $ A_{b_i +1} $ is in the thick subcategory generated by $ \bigoplus_{t \in T} A_{b_i}(t) $.
	\begin{equation*} 
		h_\pol(\twist) \leq \ell-h.
	\end{equation*} 
\end{prop}
\begin{proof} 
	Follows from Proposition \ref{prop:complexitybootstrap_refined} and the construction of a pyramid in Lemma \ref{lemma:towertopyramid}.
\end{proof}
More generally, we have the
\begin{prop}
	Let $ A $ be a nontrivial connected, graded, cocommutative Hopf algebra which is finite-dimensional over a field $ k $. 
	Suppose given a pyramid $ (N_1, \ldots, N_{\ell}) $ in $ \StMod_A $ such that $ N_1 $ is periodic, i.e. $ N_1(a)[b] \simeq N_1 $ where $ a, b $ are integers, $ a \neq 0 $. 
	Assume that there exist integers $ b_i $, $ 1 \leq b_1 < \cdots b_h < \ell $, such that $ N_{b_i +1} $ is in the thick subcategory generated by $ \bigoplus_{t \in T} N_{b_i}(t) $.
	Then for $ n \gg 0 $, 
	\begin{equation*} 
		h_\pol(\twist) \leq \ell-1- h. 
	\end{equation*} 
\end{prop}
\begin{proof}
	This follows from a special case of Proposition \ref{prop:complexitybootstrap_refined}. 
\end{proof}
The next proposition provides a criterion under which the assumption of Proposition \ref{prop:upperboundres_length_refined} holds. 
While we do not verify the assumptions of Propositions \ref{prop:nilextensions_cohomology} for the examples in \S\ref{section:examples}, we find this an instructive result concerning the relationship between non-nilpotent self-maps of nonzero internal degree (corresponding to polynomial generators of $ \Ext_A(N,N) $) and optimality of inductive bounds obtained in Propositions \ref{prop:complexitybootstrap} and \ref{prop:upperboundres_length}.   
\begin{prop}\label{prop:nilextensions_cohomology}
	Let $ A $ be a connected, graded, cocommutative Hopf algebra which is finite-dimensional over a field $ k $. 
	Let a tower under $ A $ (Definition \ref{defn:hopfalg_tower}) be given.  
	Suppose that the image of the map (\ref{eq:proto_inductive_periodicity_seq2})
	\begin{align*}
		&\Ext_{M_i}(k, k) \to \Ext_{A_i}(A_{i-1}, A_{i-1}) \\
		&(k \to \Sigma^{2}k) \mapsto (A_{i-1} \to \Sigma^{2} A_{i-1})
	\end{align*}
	classifying the resolution (\ref{eq:inductive_periodicity_sequences}) is nilpotent. 
	Then $ A_{i} $ is in the thick subcategory generated by $ \bigoplus_{j \in S} A_{i-1}(j) $ for some finite subset $ S \subset \Z $.  
\end{prop}
\begin{proof}
	Follows from Proposition \ref{prop:nilextensions}.  
\end{proof}

\section{Examples \& Computations}\label{section:examples}
\inLongVersion{\subsection{Via the results of \S\ref{section:twist_entropy_results}.}}{\subsection{Commutative Hopf algebras.}}

The simplest cases occur when our Hopf algebra is bicommutative. 
The next example is an alternative perspective on Proposition \ref{prop:entropy_twistcohproj_inspiration}. 
\begin{ex}
	[Exterior Hopf algebras] \label{ex:exteriorhopfalg_hpol_computed}
	Let $ A = \bigwedge_k[e_1, \ldots, e_d] $ be the exterior Hopf algebra where each $ e_i $ is primitive. 
	We can write down a tower (\ref{eq:hopfalg_tower}) for $ A $ with $ M_i = \bigwedge_k[e_{i+1}] $ and $ A_i = \bigwedge_k[e_{i+1}, \ldots, e_d] $ terminating at $ A_{d} $. 
	Then Proposition \ref{prop:upperboundres_length} implies that the shift functor on the stable module category of $ A $ has categorical polynomial entropy $ h_{\pol}(\twist) \leq d-1 $. 
	On the other hand, the cohomology ring of $ A $ is $ \pi_{*,*}\Ext_A(k, k) = k[x_i] $ where $ |x_i| = (1, |e_i|+1) $ by Example \ref{ex:exteriorhopf_cohomology_computed}. 
	Thus by Proposition \ref{prop:krulldimlowerbound}, $ h_{\pol}(\twist) \geq \dim H^*(A; k) - 1 = d-1 $. 
	Taken together, these imply
	\begin{equation*}
		h_\pol(\twist) = d-1 .
	\end{equation*}
\end{ex}
Notice that the tower of Example \ref{ex:exteriorhopfalg_hpol_computed} arises from a tensor product decomposition of $ A $. 
The following theorem (due to Hopf  \cite{MR4784} in characteristic zero and Borel \cite[Théorème 6.1]{MR51508} in positive characteristic) allows us to generalize the Example \ref{ex:exteriorhopfalg_hpol_computed} to arbitary finite graded primitively-generated bicommutative Hopf algebras over a field.  
\begin{theorem}\label{thm:borel_hopfalg_structure}
	Let $ A $ be a bicommutative graded finite-type Hopf algebra over a perfect field $ k $. 
	Then there exist Hopf algebras $ M_i $ where each $ M_i $ are monogenic such that $ A = \bigotimes_{i=1}^n M_i $ as Hopf algebras.  
\end{theorem}
\begin{cor}
	Let $ A $ be a bicommutative connected graded finite Hopf algebra over a perfect field $ k $ of arbitrary characteristic. 
	Let $ A = \bigotimes_{i=1}^n M_i $ be a decomposition of $ A $ into monogenic Hopf algebras $ M_i $. 
	Then the categorical polynomial entropy of the twist functor $ \twist:\StMod_A \to \StMod_A $ is given by
	\begin{equation*} 
		 h_\pol(\twist) = n - 1 . 
	\end{equation*}
\end{cor}
\begin{proof}
	Since $ A $ is finite, each $ M_i $ is finite, so we can take generators $ x_i \in M_i $ of finite height $ h_i >0 $ and positive graded degree $ |x_i| > 0 $. 
	Then 
	\begin{equation*} 
		\cdots \to M_i(h|x_i|) \xrightarrow{\cdot x_i^{h-1}} M_i(|x_i|) \xrightarrow{\cdot x_i} M_i \longrightarrow k 
	\end{equation*} 
	is a minimal resolution of $ k $ by projective $ M_i $-modules. 
	It follows that $ H^*(M_i; k) \simeq k[y] \otimes_k \bigwedge_k(z) $ where $ |y| = (2, h|x_i|) $ and $ |z| = (1, |x_i|) $ by Proposition \ref{prop:minres_ext}.  
	Then by Observation \ref{obs:Kunneththm}, $ \mathrm{Krull\; dim}\; H^*(A; k) = n $. 
	The inequality $ h_\pol(\twist) \geq n -1 $ follows Proposition \ref{prop:krulldimlowerbound}.   
	The inequality $ h_\pol(\twist) \leq n -1 $ follows from combining Lemma \ref{lemma:periodic_monogenic_modules} and Proposition \ref{prop:upperboundres_length} applied to a tower (\ref{eq:hopfalg_tower}) for $ A $ constructed analogously to that of Example \ref{ex:exteriorhopfalg_hpol_computed}. 
\end{proof}

\subsection{Noncommutative Hopf algebras.}\label{subsection:noncomm_polentropy_computed} 
Typical Hopf algebras in homotopy theory arising as algebras of cohomology operations (see Introduction to \S\ref{section:hopfalgebras}) are cocommutative but not bicommutative. 
We consider two examples of non-commutative Hopf algebras in this section. 
A central theme of this section is the importance of choosing an appropriate pyramid (Definition \ref{defn:ladder_pyramid}). 
In the following examples, we show that an obvious choice for a pyramid does not lead to an optimal bound for $ h_\pol(\twist) $.

We begin with a small example of a cocommutative but noncommutative Hopf algebra native to stable homotopy theory. 
\begin{ex}\label{ex:SteenrodA1_polentropy_computed}
Consider the finite subalgebra $ \mathcal{A}_1 $ of the mod 2 Steenrod algebra of Example \ref{ex:mod2steenrod}. 
An arguably natural choice of pyramid for $ \StMod_{\mathcal{A}_1} $ can be had by applying the strategy of Proposition \ref{prop:upperboundres_length} to the extension in Example \ref{ex:SteenrodA1_extension}. 
We obtain a tower with $ M_0 = \bigwedge_{\F_2}[Q_1] $, $ M_1 = \bigwedge_{\F_2}[\Sq^1] $, $ M_2 = \bigwedge_{\F_2}[\Sq^2] $. 
Lemma \ref{lemma:towertopyramid} furnishes a pyramid in $ \StMod_{\mathcal{A}_1} $ of dimension $ 3 $. 
Then Proposition \ref{prop:upperboundres_length} implies that the categorical polynomial entropy of the twist functor $ \twist: \StMod_{\mathcal{A}_1} \to \StMod_{\mathcal{A}_1} $ is bounded from above by $ h_{\pol}(\twist) \leq 2 $. 

However, Proposition \ref{prop:krulldimlowerbound} and Example \ref{ex:SteenrodA1_cohom_Krulldim} imply that $ h_{\pol}(\twist) \geq 1 $. 
This discrepancy between the lower and upper bounds for $ h_{\pol}(\twist) $ can be resolved by identifying an explicit pyramid in $ \StMod_{\mathcal{A}_1} $. 
\end{ex}
\begin{prop}
	Let $ \mathcal{A}_1 $ be the finite subalgebra of the mod 2 Steenrod algebra of Example \ref{ex:mod2steenrod}. 
	Then the categorical polynomial entropy of the twist functor $ \twist: \StMod_{\mathcal{A}_1} \to \StMod_{\mathcal{A}_1} $ is bounded from above by $ h_{\pol}(\twist) \leq 1 $. 
\end{prop}
\begin{proof}
	Follows from Proposition \ref{prop:upperboundres_length} and observing that Lemma \ref{lemma:good_bootstrap_SteenrodA1} furnishes a pyramid of dimension 2. 
\end{proof}
The preceding example is topological in nature: To see this let $ BO $ be the classifying space of the infinite orthogonal group over $ \R $. 
This space represents the cohomology theory which assigns to a compact space $ X $ the group completion of the topological monoid of $ \R $ vector bundles on $ X $. 
Then a result of Stong \cite{MR151963} shows that the cohomology of $ BO $ is isomorphic to the quotient Hopf algebra $ H^*(BO; \F_2) \simeq \mathcal{A} \otimes_{\mathcal{A}_1} \F_2 $ as modules over the mod 2 Steenrod algebra. 

\begin{lemma}\label{lemma:good_bootstrap_SteenrodA1}
	Let $ \mathcal{A}_1 $ be the sub-Hopf algebra of the mod 2 Steenrod algebra generated by $ \Sq^1 $ and $ \Sq^2 $. 
	Let $ B_0 $ be the left ideal of $ \mathcal{A}_1 $ generated by $ \Sq^1 $. 
	Write $ p:B_0 \to \F_2 $ for the projection map which mods out by $ \Sq^2 $. 
	Let $ B_1 $ be the quotient $ \mathcal{A}_1/(\Sq^1, \Sq^1 \Sq^2) $. 
	Then the following hold in $ \StMod_{\mathcal{A}_1} $:
	\begin{itemize}
		\item There is an exact sequence
		\begin{align*}
		\begin{tikzcd}[ampersand replacement=\&,row sep=tiny]
			\F_2(12) \ar[r] \& B_0(6) \ar[r] \& B_0(-1)[-2] \ar[r,"p"] \& \F_2[-2]
		\end{tikzcd}
		\end{align*}
		\item The module $ B_0 $ is periodic: $ B_0(-1)[1] \simeq B $.
	\end{itemize}
\end{lemma}
\begin{proof}
	We construct the first exact sequence; the existence of the second exact sequence is straightforward. 

	As a submodule of $ \mathcal{A}_1 $, $ B_0 $ is generated as a $ \F_2 $-vector space by the elements
	\begin{equation*} 
		\Sq^1, \Sq^2\Sq^1, \Sq^1\Sq^2\Sq^1, \Sq^2\Sq^1\Sq^2\Sq^1 .
	\end{equation*} 
	The map $ p_0 := p: B_0(-1) \to k(0) $ is given by projection onto $ \Sq^1 $. 
	Consider the $ \mathcal{A}_1 $-linear morphism $ p_1: \mathcal{A}_1(2) \to \ker(p_0) $ which sends $ 1 \mapsto \Sq^2\Sq^1 $. 
	By explicit calculation using the Adem relations, we see that
	\begin{itemize}
		\item $ p_1 $ surjects onto the kernel of $ p_0 $ 
		\item The kernel of $ p_1 $ is the left ideal of $ \mathcal{A}_1 $ generated by $ \Sq^2 $. 
		Explicitly, $ \ker(p_1) $ is generated as a $ k $-vector space by $ \Sq^2, \Sq^1\Sq^2, \left(\Sq^2\right)^2, \Sq^2\Sq^1\Sq^2, \left(\Sq^2\right)^3= \left(\Sq^2\Sq^1\right)^2 = \left(\Sq^1\Sq^2\right)^2 $. 
	\end{itemize}
	By Observation \ref{obs:stmod_loops}, $ \ker(p_1) \simeq \loops \ker(p_0) $ in $ \StMod_A $. 

	Now consider the $ \mathcal{A}_1 $-linear morphism $ p_2 \colon \mathcal{A}_1(4) \to \ker(p_1) $ which sends $ 1 \mapsto \Sq^2 $. 
	By explicit calculation using the Adem relations (\ref{eq:ademrelations}), we see that
	\begin{itemize}
		\item $ p_2 $ surjects onto the kernel of $ p_1 $ 
		\item The kernel of $ p_2 $ is the left ideal of $ \mathcal{A}_1 $ generated by $ \Sq^1 \Sq^2 $. 
		Explicitly, $ \ker(p_2) $ is generated as a $ k $-vector space by $ \Sq^1\Sq^2, \Sq^2\Sq^1\Sq^2, \left(\Sq^2\right)^3= \left(\Sq^2\Sq^1\right)^2 = \left(\Sq^1\Sq^2\right)^2 $. 
	\end{itemize}
	By Observation \ref{obs:stmod_loops}, $ \ker(p_2) \simeq \loops \ker(p_1) $. 

	Finally, there is a surjective $ \mathcal{A}_1 $-linear map $ p_3: B_0(6) \to \ker(p_2) $ which takes $ \Sq^1 \mapsto \Sq^1\Sq^2 $. 
	The result follows from observing $ \ker(p_3) \simeq \F_2(12) $ and splicing two exact sequences together. 
\end{proof}

\begin{ex}\label{ex:noncommex}
	Define $ M $ to be the dual of the commutative coassociative $ 8 $-dimensional Hopf algebra $ M^\vee = \bigwedge_{\F_2}(e_1, e_2, e_3) $ where $ e_i $ lives in grading degree $ -i $. 
	Define the comultiplication such that $ e_1 $ and $ e_2 $ are primitive and
	\begin{align*}
		\Delta(e_3) = 1 \otimes e_3 + e_1 \otimes e_2 + e_3 \otimes 1 .
	\end{align*}
	Passing to the dual, this says that $ M $ is cocommutative and, as an associative algebra, is generated by elements $ x_1, x_2, x_3 $ where $ x_i $ lives in grading degree $ i $. 
	The $ x_i $ satisfy the relations
	\begin{itemize}
		\item $ x_i^2 = 0 $ for $ i = 1, 2, 3$. 
		\item The element $ x_3 $ commutes with both $ x_1 $ and $ x_2 $.
		\item There is a nontrivial commutation relation $ x_1 x_2 = x_2 x_1 + x_3 $. 
	\end{itemize}
	Observe that taking $ M_0 = \bigwedge_{\F_2}(e_3) $ and $ M_i = \bigwedge_{\F_2}(e_i) $ for $ i = 1, 2 $ produces a tower (\ref{eq:hopfalg_tower}) for $ M $, giving the naïve bounds  
	\begin{align*}
		h_{cat}(\twist) &= 0 \\
		h_\pol(\twist) &\leq 2 .
	\end{align*} 
	On the other hand, Propositions \ref{prop:krulldimlowerbound} and \ref{prop:noncommex_cohkrulldim_computed} imply that $ h_\pol(\twist) \geq 1 $. 
	We resolve this discrepancy by exhibiting a pyramid in $ \StMod_M $ as follows. 
	Let $ X,Y $ be the \emph{left} ideals generated by $ x_1, x_2 $ respectively. 
	Consider the exact sequences 
	\begin{align}
	\begin{tikzcd}[ampersand replacement=\&]
		X(1) \ar[r] \& M \ar[r] \& X 
	\end{tikzcd} \label{eq:noncommex_Xper}
	\\ 
	\begin{tikzcd}[ampersand replacement=\&] \label{eq:noncommex_Yper}
		Y(2) \ar[r] \& M \ar[r] \& Y
	\end{tikzcd}
	\\ 
	\begin{tikzcd}[ampersand replacement=\&]  \label{eq:noncommex_ladder}
		k(6) \ar[r] \& M \ar[r,"{\cdot x_1 \oplus \cdot x_2}"] \& X(-1) \oplus Y(-2) \ar[r] \& k(0)
	\end{tikzcd}
	\end{align}
	in $ \Mod_M $. 
	In particular, the non-commutativity of $ M $ is what makes the third sequence above exact in grading degree 3: the center morphism in this degree is given by
	\begin{align*}
		M_3 &\to X_4 \oplus Y_5 \\
		x_3 &\mapsto (x_3 x_1,x_3 x_2) \\
		x_1x_2 &\mapsto (x_1x_2 \cdot x_1,x_1x_2 \cdot x_2) = ( x_1x_3, 0).
	\end{align*}
	We observe that the sequences (\ref{eq:noncommex_Xper}) and (\ref{eq:noncommex_Yper}) imply that $ X(-1) \oplus Y(-2) $ is periodic, while the sequence (\ref{eq:noncommex_ladder}) passes to an exact sequence
	\begin{equation*}
	\begin{tikzcd}[ampersand replacement=\&]
		k(6)[-1] \ar[r] \& X(-1) \oplus Y(-2) \ar[r] \& k(0)
	\end{tikzcd}
	\end{equation*}
	in $ \StMod_M $.
	Therefore, $ (X(-1) \oplus Y(-2), k(0)) $ is a pyramid for $ M $. 

	An application of the arguments of Theorem \ref{thm:twist_expentropy_iszero} and Proposition \ref{prop:upperboundres_length} imply that the twist functor $ \twist: \StMod_M \to \StMod_M $ satisfies $ h_\pol(\twist \leq 1 $. 
	Taken together, we have shown that 
	\begin{align*}
		h_\pol(\twist) = 1 .
	\end{align*} 
\end{ex}
\begin{rmk}
	One should contrast the preceding example with the special case of Example \ref{ex:exteriorhopfalg_hpol_computed} where $ d = 3 $ and $ e_i = i $. 
\end{rmk}
\begin{prop}\label{prop:noncommex_cohkrulldim_computed}
	Let $ M $ be the connected graded cocommutative Hopf algebra over $ \F_2 $ of Example \ref{ex:noncommex}. 
	Then the cohomology of $ M $ has Krull dimension
	\begin{equation*}
		\dim H^*(M; \F_2) = 2 .
	\end{equation*}
\end{prop}
\begin{proof}
	Consider the central extension of Hopf algebras
	\begin{equation*}
	\begin{tikzcd}
		\bigwedge_{\F_2}[e_3] \ar[r] &  M \ar[r] & \bigwedge_{\F_2}[e_1, e_2]. 
	\end{tikzcd}
	\end{equation*}
	By Proposition \ref{prop:CE_sseq_coh}, there is a spectral sequence of trigraded algebras with $ E_2 $-page 
	\begin{align*}
		E_2^{p,q} = \Ext_{\bigwedge_{\F_2}[e_1, e_2]}^{^p,*}(\F_2, \F_2) \otimes \Ext_{\bigwedge_{\F_2}[e_3]}^{q,*}(\F_2, \F_2) \simeq \F_2[\alpha_1, \alpha_2, \alpha_3] \\
		|\alpha_1| = (1, 0, 1) \qquad |\alpha_2| = (1, 0, 2) \qquad |\alpha_3| = (0, 1, 3)
	\end{align*} 
	converging to $ \Ext_M^{s,t}(\F_2, \F_2) $. 
	Since $ d^2(e_3) = e_1 \otimes e_2 $ in the cobar complex of $ M $ (Construction \ref{cons:cbar_aug_filtration}), $ \alpha_3 $ cannot survive the spectral sequence. 
	By degree considerations,  
	\begin{align*}
		d_2 \colon & \alpha_3 \mapsto e_1 |e_2 = \alpha_1 \cdot \alpha_2 . 
	\end{align*}
	Since the differential is a derivation (Proposition \ref{prop:CE_sseq_coh}(\ref{propitem:CE_sseq_existence})), we have
	\begin{equation*}
		E_3 = \F_2[\alpha_1, \alpha_2]/(\alpha_1\alpha_2) \otimes_{\F_2} \F_2[(\alpha_3)^2]
	\end{equation*}
	By the Kudo transgression theorem (Proposition \ref{prop:CE_sseq_coh}(\ref{propitem:CEsseq_kudo})) and Cartan formula for Steenrod operations (Proposition \ref{prop:steenrod_hopf_coh}), we have
	\begin{align*}
		d_3 \colon \alpha_3^2 = \Sq^1\alpha_3 \mapsto \Sq^1( \alpha_1 \alpha_2) = \Sq^0( \alpha_1) \Sq^1(\alpha_2) + \Sq^1( \alpha_1) \Sq^0(\alpha_2)
	\end{align*}
	where the right-hand side vanishes because $ \Sq^0(\alpha_i) = e_i^2 = 0 $ for $ i = 1 ,2 $. 
	Then the Kudo transgression theorem implies that the rest of the differentials vanish, so the spectral sequence collapses and $ E_3 = E_\infty $. 
	The resulting ring has Krull dimension 2. 
\end{proof}

\appendix
\section{Graded Schwede-Shipley}
The goal of this section is to prove a graded enhancement of \cite[Theorem 3.1.1]{MR1928647}. 
Recall our notation $ \Z^\delta $ for the discrete category (no nonidentity morphisms) with objects given by the integers. 

\subsection{Generalities on presentable \texorpdfstring{$ \infty $}{∞}-categories.} 
Write $ \Pr^L $ for the $ \infty $-category of presentable $ \infty $-categories with colimit-preserving functors. 
We refer readers to \cite[Definition 5.5.0.1]{LurHTT} for the precise definition of a presentable $ \infty $-category; the main point is that a presentable category $ \cat $, while necessarily being large, is generated under colimits by a small collection of objects.  
\begin{recollection}
	[Lurie tensor product] The category $ \Pr^L $ admits a symmetric monoidal structure \cite[Proposition 4.8.1.15]{HA}. 
	Given $ \cat, \mathcal{D} \in \Pr^L $, their tensor product $ \cat \otimes \mathcal{D} $ is initial among presentable $ \infty $-categories receiving a functor from $ \cat \times \mathcal{D}$ which preserves small colimits separately in each variable. 
\end{recollection}
\begin{defn} \cite[\S2.2]{Lurie-Rot}
	Let $ \Spectra^\gr := \Fun(\Z^\delta, \Spectra) $ be the category of \emph{graded spectra}. 

	Let $ \Spc_*^\gr := \Fun(\Z^\delta,\Spc_*) $ be the category of \emph{graded pointed spaces}. 


\end{defn}

\begin{obs}
	\label{obs:graded_spectra_dayconv_pres}
	By \cite[Proposition 5.5.3.6]{LurHTT}, $ \Spectra^\gr $ and $ \Spc^\gr $ are presentable. 
	The categories $ \Spectra^\gr $ and $ \Spc^\gr $ are $\E_1 $-algebra objects in $ \Pr^L $ via Day convolution \cite[Example 2.2.6.17]{HA}. 
	Moreover, the levelwise suspension spectrum functor $ \Sigma^\infty_+: \Spc^\gr \to \Spectra^\gr $ exhibits $ \Spectra^\gr $ as an algebra over $ \Spc^\gr $. 
\end{obs}
The following definition is a ``large'' analogue of \cite[Notation 2.4.10]{Lurie-Rot}.
\begin{defn}\label{defn:loc_gr_pres_cats}
	Using Observation \ref{obs:graded_spectra_dayconv_pres}, we define the categories of \emph{locally graded presentable $ \infty $-categories} and of \emph{locally graded presentable stable $ \infty $-categories}: 
	\begin{align*}
		\mathrm{Pr}^{L,\gr} &:= \Mod_{\Spc^{\gr}}(\mathrm{Pr}^{L}) \\
		\mathrm{Pr}^{L,\mathrm{st},\gr} &:= \Mod_{\Spectra^{\gr}}(\mathrm{Pr}^{L,\mathrm{st}}). 
	\end{align*}
	In particular, we will refer to a morphism $ F: \cat \to \mathcal{D} $ in $ \mathrm{Pr}^{L,\gr} $ as a \emph{graded functor}. 
\end{defn}
Given a graded presentable $ \infty $-category $ \cat $, restricting the structure morphism $ \Spectra^\gr \otimes \cat \to \cat $ to the full subcategory $ (n): \Spectra \inj \Spectra^{\gr} $ on spectra concentrated in degree $ n $ furnishes a functor $ (n): \cat \to \cat $. 
In particular, $ (n) $ is an autoequivalence of $ \cat $ (cf. \cite[Definition 2.4.2]{Lurie-Rot}). 
\begin{ex}
	Let $ \cat $ be any presentable $ \infty $-category. 
	Then $ \cat^\gr:= \Fun_{\Cat_\infty}(\Z^\delta, \cat) $ is a graded presentable $ \infty $-category. 
	If $ \cat $ is furthermore stable, then $ \cat^\gr $ is a graded presentable stable $ \infty $-category.
\end{ex}
\begin{ex}\label{ex:gradedmod_gradedring}
	Let $ R \in \E_1\Alg(\Spectra^\gr) $. 
	Then the category of $ R $-modules in graded spectra $ \Mod_R(\Spectra^{\gr}) $ is a graded presentable $ \infty $-category. 
\end{ex}
The goal of this section is to show that a class of sufficiently nice graded categories are of the form of Example \ref{ex:gradedmod_gradedring}. 
\begin{obs}
	The category $ \mathrm{Pr}^{L,\gr} $ inherits a symmetric monoidal structure from $ \Pr^{L} $ by \cite[Theorem 4.5.2.1]{HA}. 
	The tensor product of two graded presentable stable $ \infty $-categories $ \cat,\mathcal{D} $ can be identified with their relative tensor product over $ \Spc^\gr $. 
\end{obs}
For an $ \infty $-category $ \cat $, being stable is a property, not additional structure. 
Reinterpreting stability as admitting a module structure over the $ \infty $-category of spectra, this is captured by the following definition from \cite[\S4.8.2]{HA}. 
\begin{defn}
	Let $ \cat $ be a symmetric monoidal $ \infty $-category, and let $ \mathbbm{1} $ denote the $ \otimes $-unit of $ \cat $. 
	A morphism $ e: \mathbbm{1} \to E $ is an \emph{idempotent object} of $ \cat $ if the composites
	\begin{equation*}
		\id_E \otimes e: E \to E \otimes E \qquad \qquad e \otimes \id_E : E \to E \otimes E
	\end{equation*}
	are equivalences in $ \cat $. 

	A pair $ (\mathcal{D}, X) $ where $ \mathcal{D} $ is a presentable $ \infty $-category and $ X \in \mathcal{D} $ is \emph{idempotent} if there exists a colimit-preserving functor $ F: \Spc \to \mathcal{D} $ with $ F(*) = X $ exhibiting $ \mathcal{D} $ as idempotent in $ \mathrm{Pr}^L $ \cite[discussion before Proposition 4.8.2.11]{HA}. 
\end{defn}
The following is a mild generalization of \cite[Proposition 4.8.2.18]{HA}. 
Write $ \sph^0 \in \Spectra $ for the sphere spectrum.
\begin{prop}\label{prop:gradedspectra_idempotent}
	The pair $ (\Spectra^\gr, \sph^0(0)) $ is idempotent in $ \mathrm{Pr}^{L,\gr}  $. 
\end{prop}
\begin{cor}
	The forgetful functor $ \mathrm{Pr}^{L,\stable \gr} \to  \mathrm{Pr}^{L,\gr} $ exhibits the former as a full subcategory of the latter.
\end{cor}
\begin{proof}
	This follows from \cite[Proposition 4.8.2.4(3)]{HA} and the definition of a localization functor \cite[Definition 5.2.7.2]{LurHTT}. 
\end{proof}
\begin{proof} [Proof of Proposition \ref{prop:gradedspectra_idempotent}]
	This follows from taking $ \cat = \mathrm{Pr}^L $ and $ A = \Spc^\gr $ and $ E = \Spectra $ in Lemma \ref{lemma:idempotent_tensor}. 
	By \cite[Proposition 4.8.1.7]{HA},
	\begin{align*}
		\Spc^\gr \otimes \Spectra \simeq \mathrm{R}\Fun(\Spc^{\gr,\op},\cat) \simeq \Spectra^\gr
	\end{align*}
	where the second equivalence follows from noting that $ \Spc^\gr = \Fun(\Z^\delta, \cat) $ as a presheaf category is freely generated by $ \Z^\delta $ under homotopy colimits \cite[Theorem 5.1.5.6]{LurHTT}. 
\end{proof}
\begin{lemma}\label{lemma:idempotent_tensor}
	Let $ \cat $ be a symmetric monoidal $ \infty $-category, and let $ A $ be an $ \E_\infty $-algebra object in $ \cat $. 
	Let $ E \in \cat $, and suppose $ e: \mathbbm{1} \to E $ exhibits $ E $ as an \emph{idempotent object} in $ \cat $. 
	Then 
	\begin{equation*} 
		e \otimes \id_A: A \to E \otimes A 
	\end{equation*}
	exhibits $ E \otimes A $ as an idempotent object in $ \Mod_A(\cat) $. 
\end{lemma}
\begin{proof}
	By the characterization of idempotence given by Remark 4.8.2.6 of \cite{HA}, it suffices to show that the map 
	\begin{equation*} 
		(e \otimes \id_A) \otimes_A \id_{E \otimes A}: E \otimes A \to (E \otimes A) \otimes_A (E \otimes A)  
	\end{equation*}
	is an equivalence. 
	Rearranging, we find the previous is equivalent to $ \id_A:A = A $ tensored with $ e \otimes \id_E: E \to E \otimes E $, which is an equivalence by our assumption. 
\end{proof}

\subsection{The main theorem.}
The following is a mild generalization of \cite[Theorem 7.1.2.1]{HA}.
\begin{prop}\label{prop:gradedschwedeshipley}
	Let $ \cat $ be a graded stable $ \infty $-category. 
	Then $ \cat $ is equivalent to $ \LMod_R^\gr $ for some graded $ \E_1 $-ring $ R $ if and only if $ \cat $ is presentable and there exists a compact object $ G \in \cat $ which graded-generates $ \cat $ in the following sense: 
	If $ M \in \cat $ is an object having the property that $ \Ext^n_\cat(G, M) \simeq 0 $ for all $ n $ as a graded spectrum (Recollection \ref{rec:graded_vs}\ref{recitem:graded_enrichment}), then $ M \simeq 0 $. 
\end{prop}
\begin{proof}
	We begin with a preliminary observation: There is a commutative diagram \cite[Construction 4.8.3.4]{HA}
	\begin{equation*}
	\begin{tikzcd}[column sep=tiny]
		\Cat_\infty^\Alg(\kappa) \ar[rr,"\Theta"] \ar[rd,"U"'] & & \Cat^{\Mod}_\infty(\kappa)_{\mathfrak{M}/-} \ar[ld] \\
		& \Mon_{\E_1}^\kappa(\Cat_\infty)
	\end{tikzcd} .
	\end{equation*}
	We will regard $ \Spectra^\gr \in \Mon_{\E_1}^\kappa(\Cat_\infty) $.  
	By definition the fiber over $ \Spectra^\gr $ of $ U $ is equivalent to $ \Alg_{\E_1}(\Spectra^\gr) $.  
	Then it follows from Theorem 4.8.5.5 that $ \Theta $ restricts to an equivalence
	\begin{equation*}
		\Alg_{\E_1}(\Spectra^{\gr}) \to \Cat^{\Mod}_{\infty}(\kappa)_{\Spectra^{\gr}/-} .
	\end{equation*}
	Suppose $ \cat $ is locally graded presentable. 
	Then by Proposition \ref{prop:gradedspectra_idempotent}, we can regard $ \cat $ as a module over $ \Spectra^\gr $ in $ \Pr^{L,\gr} $. 
	There is a colimit-preserving graded functor
	\begin{align*} 
		F\colon &\Spectra^\gr \to \cat \\
		&X_\bullet \mapsto X_\bullet \otimes G .
	\end{align*} 
	By \cite[Corollary 5.5.2.9]{LurHTT}, $ F $ admits a graded right adjoint $ R $. 
	We will show that $ \left(\cat,  G\right) $ lies in the essential image of the fully faithful functor $ \Alg_{\E_1} \left(\Spectra^\gr \right) \to \Mod_{\Spectra^\gr}(\mathrm{Pr}^L)_{\Spectra^\gr/-} $ of \cite[Theorem 4.8.5.5]{HA}. 
	By \cite[Proposition 4.8.5.8]{HA}, it suffices to show the following:
	\begin{enumerate}[label=(\alph*)]
		\item \label{wtsitem:G_pres_geomreal} The functor $ R $ preserves geometric realizations of simplicial objects. 
		\item \label{wtsitem:G_cons} The functor $ R $ is conservative. 
		\item \label{wtsitem:G_tensor_commute} For every object $ M \in \cat $ and every spectrum $ Y \in \Spectra $, the canonical map $ {\theta_{Y,M}: Y \otimes G(M) \to G(Y \otimes M)} $ is an equivalence. 
	\end{enumerate}
	To prove \ref{wtsitem:G_pres_geomreal}, it suffices to show that $ R $ preserves all small colimits. 
	Since $ R $ is exact, it suffices to show that $ R $ preserves small filtered colimits, which follows from the fact that $ G $ is compact. 
	To prove \ref{wtsitem:G_cons}, suppose we are given a map $ \alpha: N \to N' $ such that $ R(\alpha) $ is an equivalence, and let $ N'':= \cofib (\alpha) $. 
	Then $ G(N'') \simeq 0 $, so $ \pi_n \hom\left(G, N'' \right) $ vanishes \emph{as a graded spectrum} (see Recollection \ref{rec:graded_vs}\ref{recitem:graded_enrichment}) for all $ n $. 
	Our assumption that $ \left\{ G\right\}_{\ell \in \Z} $ generate $ \cat $ implies that $ N'' \simeq 0 $, so $ \alpha $ is an equivalence. 

	To prove \ref{wtsitem:G_tensor_commute}, fix an $ M \in \cat $ and consider the collection $ \mathcal{E} $ of spectra $ Y $ such that $ \theta_{M, Y} $ is an equivalence. 
	Note that $ \mathcal{E} $ is stable/closed under suspension and desuspension and closed under colimits, so it suffices to show that $ \sph^0 \in \mathcal{E} $, which is immediate. 
\end{proof}

Next, we prove a monoidal enhancement of Proposition \ref{prop:gradedschwedeshipley}, or a graded enhancement of \cite[Proposition 7.1.2.6]{HA}. 
\begin{prop}
	Let $ k \in \Z_{>0} \cup \{\infty \} $. 
	The construction $ R \mapsto \LMod^\gr_R $ determines a fully faithful embedding
	\begin{equation*}
		\Alg_{\E_k}(\Spectra^\gr) \to \Alg_{\E_{k-1}}\left(\mathrm{Pr}^{L,\gr}\right)
	\end{equation*}
	from the category of graded $ \E_k $-algebras in spectra to the $ \infty $-category of graded $ \E_{k-1} $-monoidal presentable $ \infty $-categories. 
	An $ \E_k $-monoidal graded $ \infty $-category $ \cat^\otimes \to \E_k^\otimes $ belongs to the essential image if and only if
	\begin{enumerate}
		\item The $ \infty $-category is stable and presentable, and if $ k>1 $ the tensor product preserves small colimits separately in each variable. 
		\item The unit object $ \mathbbm{1} \in \cat $ is compact. 
		\item The object $ \mathbbm{1} $ graded-generates $ \cat $ in the following sense: 
		If $ M \in \cat $ is an object having the property that $ \Ext^n_\cat(\mathbbm{1}, M) \simeq 0 $ for all $ n $ as a graded spectrum (Recollection \ref{rec:graded_vs}\ref{recitem:graded_enrichment}), then $ M \simeq 0 $. 
	\end{enumerate}
\end{prop}
\begin{proof}
	Full faithfulness follows from Corollary 5.1.2.6 of \cite{HA} and Proposition \ref{prop:gradedspectra_idempotent}.
	One arrives at the description of the essential image in the same manner as in the proof of Proposition \ref{prop:gradedschwedeshipley}. 
\end{proof}

\addcontentsline{toc}{section}{\protect\numberline{\thesection}References}

\printbibliography

\end{document}